\documentclass{amsart}
\usepackage{amsmath,amssymb%,showkeys}
}\newcommand{\Alt}{\mathop{\mathrm{Alt}}}
\newcommand{\Sym}{\mathop{\mathrm{Sym}}}
\newcommand{\Aut}{\mathop{\mathrm{Aut}}}

\newcommand{\PSL}{\mathop{\mathrm{PSL}}}
\newcommand{\PGL}{\mathop{\mathrm{PGL}}}
\newcommand{\PSU}{\mathop{\mathrm{PSU}}}

\newcommand{\Out}{\mathop{\mathrm{Out}}}
\newcommand{\PSp}{\mathop{\mathrm{PSp}}}
\newcommand{\BBB}{\mathop{\mathrm{B}}}
\newcommand{\DDD}{\mathop{\mathrm{D}}}
\newcommand{\EEE}{\mathop{\mathrm{E}}}
\newcommand{\FFF}{\mathop{\mathrm{F}}}
\newcommand{\GL}{\mathop{\mathrm{GL}}}
\newcommand{\GU}{\mathop{\mathrm{GU}}}
\newcommand{\SU}{\mathop{\mathrm{SU}}}
\newcommand{\G}{\mathop{\mathrm{G}}}
\newcommand{\OO}{\mathop{\mathrm{O}}}
\newcommand{\Sp}{\mathop{\mathrm{Sp}}}
\newcommand{\Gal}{\mathop{\mathrm{Gal}}}
\newcommand{\SL}{\mathop{\mathrm{SL}}}

\newcommand{\POmega}{\mathop{\mathrm{P}\Omega}}

\renewcommand{\wr}{\mathop{\mathrm{wr}}}
\newtheorem{theorem}{Theorem}[section]
\newtheorem{lemma}[theorem]{Lemma}

\newtheorem{corollary}[theorem]{Corollary}
\newtheorem{notation}[theorem]{Notation}
\newtheorem{remark}[theorem]{Remark}
\newtheorem{proposition}[theorem]{Proposition}
\newtheorem{definition}[theorem]{Definition}

\newtheorem*{conjAA}{Theorem~A}
\def\cent#1#2{{\bf C}_{#1}(#2)}
\def\N#1#2{{\bf N}_{#1}(#2)}
\def\Z#1{{\bf Z}(#1)}
\def\E#1{{\bf E}(#1)}
\def\F#1{{\bf F}(#1)}
\def\FF#1{{\bf F^*}(#1)}

\begin{document}

\title[Coprime subdegrees]{On the maximal number of coprime subdegrees in finite primitive permutation groups}   

\author[S.~Dolfi]{Silvio Dolfi}
\address{Silvio Dolfi, Dipartimento di Matematica ``Ulisse
  Dini'', Universit\'a degli studi di Firenze,\newline Viale Morgagni 67/a, 50134 Firenze Italy}
\email{dolfi@math.unifi.it}

\author[R.~Guralnick]{Robert Guralnick}
\address{Robert Guralnick, Department of Mathematics, University of
  Southern California,\newline Los Angeles, California USA 90089--2532}
\email{guralnic@usc.edu}

\author[C. E. Praeger]{Cheryl E. Praeger}
\address{Cheryl E. Praeger, Centre for Mathematics of Symmetry and Computation,\newline
School of Mathematics and Statistics,\newline
The University of Western Australia,
 Crawley, WA 6009, Australia} \email{Cheryl.Praeger@uwa.edu.au}

\author[P. Spiga]{Pablo Spiga}
\address{Pablo Spiga, Centre for Mathematics of Symmetry and Computation,\newline
 School of Mathematics and Statistics,\newline
The University of Western Australia,
 Crawley, WA 6009, Australia} \email{pablo.spiga@unimib.it}

\thanks{Address correspondence to P. Spiga,
E-mail: pablo.spiga@unimib.it\\ 
The first author is supported by the MIUR project ``Teoria dei gruppi
ed applicazioni''. The second author was partially supported by NSF
grant DMS-1001962.    
The third author is supported by the ARC Federation
Fellowship Project 
FF0776186. The fourth author is supported by the University of Western Australia
as part of the  Federation Fellowship project.}

\subjclass[2000]{20B15, 20H30}
\keywords{coprime subdegrees; primitive groups}

\begin{abstract}
%We prove that a finite primitive group has at most two non-trivial coprime subdegrees.  
The subdegrees of a transitive permutation group are the orbit lengths of a point stabilizer. For a finite primitive permutation group which is not cyclic of prime order, the largest subdegree shares a non-trivial common factor with each non-trivial subdegree. On the other hand it is possible for non-trivial subdegrees of primitive groups to be coprime, a famous example being the rank $5$ action of the small Janko group $J_1$ on $266$ points which has subdegrees of lengths $11$ and $12$. We prove that, for every finite primitive group, the maximal size of a set of pairwise coprime non-trivial subdegrees is  at most $2$.
\end{abstract}
\maketitle

\section{Introduction}\label{intro}
In this paper we are concerned with the \emph{subdegrees} of a finite primitive
permutation group. The set of  of a transitive
group $G$ is the set of orbit lengths of  the stabilizer
$G_\omega$ of a point $\omega$, and we say
that a subdegree $d$ of $G$ is \emph{non-trivial} if $d\neq 1$. We
announced in~\cite[Theorem~$1.7$]{DGPS} that a primitive
permutation group could not have as many as three pairwise coprime non-trivial subdegrees. Here
we prove this theorem. 

\begin{theorem}\label{theorem}
Let $G$ be a finite primitive permutation group. Then the largest
subset of pairwise coprime non-trivial subdegrees of $G$ has
cardinality at most $2$. 
\end{theorem}

This theorem is related to a classical result on finite primitive
groups. In~1935 Marie Weiss~\cite[Theorem~$3$]{Piter} showed
that, if $G$ is a finite primitive group which is not cyclic of prime
order, then the largest of the subdegrees has non-trivial divisors in common
with all the other non-trivial subdegrees. It was observed by Peter Neumann in 1973~\cite[Corollary~$(2)$, page~93]{Piter} that Weiss's theorem implies that a
finite primitive group with $k$ pairwise coprime non-trivial
subdegrees has rank at least $2^k$.
Neumann remarked that `groups of small rank with non-trivial co-prime subdegrees appear to be rather rare', and posed a question of Peter Cameron~\cite[Problem~$1$,  page~$93$]{Piter} on the existence of a primitive rank $4$ group with two coprime non-trivial subdegrees, that is to say, a group meeting the bound $2^k$ with $k=2$. That no such group exists was verified by Cameron himself (see~\cite[Remark in Section~$1.32$]{Peter}), using the finite simple group classification. The smallest
rank for coprime subdegrees to occur is  $5$ with the famous example of
$J_1$ of degree $266$ with subdegrees $1,11,12,110$ and $132$ first
studied by Livingstone~\cite{Livingstone}.
Our Theorem~\ref{theorem} shows that the parameter $k$ is at most $2$. We emphasise that a primitive group may have several pairs of coprime non-trivial subdegrees -- examples are given in \cite[Example 4.3]{DGPS}. Our result simply prohibits triples of pairwise coprime subdegrees.

We say that a subgroup $L$ of a
nonabelian simple group $T$ is  
\emph{pseudo-maximal} in $T$ if there exists an almost simple group $A$ with
socle $T$ and a maximal subgroup $M$ of $A$ with $T\nsubseteq M$ and $L = T\cap
M$ (see ~\cite[Definition~$1.8$]{DGPS}). Theorem~$1.10$ in~\cite{DGPS}  shows that Theorem~\ref{theorem}
holds true if the following result on nonabelian simple groups
(see~\cite[Theorem~$1.9$]{DGPS}) is true.

\begin{conjAA}\label{conj3}Let $T$ be a  transitive nonabelian simple
  permutation group and assume that the stabilizer of a point is
  pseudo-maximal in $T$. Then the largest subset of pairwise coprime
  non-trivial subdegrees of $T$ has cardinality at most $2$.
\end{conjAA}

The aim of this paper is to prove Theorem~A using the Classification of Finite Simple Groups, thus proving Theorem~\ref{theorem}. The structure of the paper is straighforward. In
Section~\ref{embedding results} and~\ref{aux} we collect some
auxiliary results. We prove Theorem~A for the alternating groups
in Section~\ref{alt}, for the classical groups in
Section~\ref{classical}, for the exceptional groups of Lie type in
Section~\ref{exceptionalLie} and  finally for the sporadic simple
groups in Section~\ref{sporadicsec}. 

One of the most efficient methods for analyzing a
finite  primitive permutation group $G$ is to study the \emph{socle} $N$ of
$G$, that is, the subgroup generated by the minimal normal subgroups
of $G$. The O'Nan-Scott theorem describes in detail the embedding of
$N$ in $G$ and collects 
some useful information on the action of $N$. In~\cite{Pr} eight types
of primitive groups are defined (depending on the structure and on the
action of the socle), namely HA (\textit{Holomorphic Abelian}), AS
(\textit{Almost Simple}), SD
(\textit{Simple Diagonal}), CD
(\textit{Compound Diagonal}), HS
(\textit{Holomorphic Simple}), HC
(\textit{Holomorphic Compound}), TW 
(\textit{Twisted wreath}), PA
(\textit{Product Action}), and it is shown in~\cite{LPS1} that every
primitive group 
belongs to exactly one of these
types. Combining~\cite[Theorem~$1.3$,~$1.4$]{DGPS} with 
Theorem~\ref{theorem}, we have the following corollary determining the
maximal number of non-trivial pairwise coprime subdegrees  of a primitive group
according to its O'Nan-Scott type. 

\begin{corollary}\label{theCor}
Let $G$ be a finite primitive group. If $G$ has two non-trivial coprime
subdegrees, then $G$ is of AS, PA or TW type. 
%Moreover, if $G$ is of AS, PA or TW type, then $G$ has at most two non-trivial coprime subdegrees.
\end{corollary}

Results concerning the subdegrees of a finite permutation group can often give interesting applications in Field Theory (see for example~\cite[Corollary~$1.9$]{DGPS}). In fact, Theorem~\ref{theorem} has the following surprising application.

\begin{corollary}\label{corA}
Let $K = k[\theta]$ be a minimal separable field extension that is not
Galois. Let $f (x) \in k[x]$ be the minimal polynomial of $\theta$ over $k$ and write $f (x) =(x - \theta)g_1 (x)\cdots g_r (x)$ with $g_i \in K[x]$  irreducible over $K$, for each $i\in \{1,\ldots,r\}$. Then the maximal number of $g_i (x)$ of pairwise coprime degree is $2$.
\end{corollary}
\begin{proof}
Let $L$ be the normal closure of $K/k$. As $K$ is separable, $L/k$ is a Galois extension. Let $G$ be the Galois group $\Gal(L/k)$ and set $H =
\Gal(L/K)$. By the minimality of $K$, the group $G$ acts primitively on the right cosets $G/H$ of $H$ in $G$. The degrees of the $g_i(x)$ are precisely the non-trivial subdegrees of $H$ in
the action on $G/H$. Now apply Theorem~\ref{theorem}.
\end{proof}

\section{Embedding results}\label{embedding results}

The main results in this section are Propositions~\ref{emb1}
and~\ref{emb2}, which will prove to be important in the proof of
Theorem~\ref{theorem}.

\begin{definition}\label{def}{\rm If $G$ is a finite group, we let $\mu(G)$
    denote the maximal size of a set $\{G_i\}_{i\in I}$ of proper
    subgroups of $G$ with $|G:G_i|$ and $|G:G_j|$ relatively prime,
    for each two distinct elements $i$ and $j$ of $I$.} 
\end{definition}
The following remark is used in Section~\ref{aux}.

\begin{remark}{\rm
The number $\mu(G)$ equals the maximal size of a set $\{M_j\}_{j\in J}$ of maximal subgroups of $G$ with $|G:M_i|$ and $|G:M_j|$ relatively prime, for each two distinct elements $i$ and $j$ in $J$. Clearly, $|J|\leq \mu(G)$. Conversely, let $\{G_i\}_{i\in I}$ be a family of  proper subgroups of $G$ with relatively prime index in $G$. Let $M_i$ be a maximal subgroup of $G$ with $A_i\leq M_i$. Since $|G:G_i| $ is coprime to $|G:G_j|$ for $i\neq j$, we have $M_i\neq M_j$ and $|G:M_i|$ is coprime to $|G:M_j|$. Thus $\mu(G)\leq |J|$.}
\end{remark}

We recall that a finite group $E$ is said to be {\em quasisimple} if
$E=[E,E]$ and $E/\Z E$ is a nonabelian simple group, where $\Z E$ is
the centre of the group $E$. Furthermore, we say that the finite group
$G$ is a {\em
  central product} of $A$ and $B$, if $A$ and $B$ are non-identity
proper subgroups of $G$
with $[A,B]=1$ and $G=AB$. We recall that a {\em component} $E$ of $G$ is a
quasisimple subnormal subgroup of $G$. As usual, we denote by $\E G$
the group generated by the set $\{E_1,\ldots,E_\ell\}$ (possibly empty) of
components of $G$, by $\F G$ the Fitting subgroup of $G$ and by $\FF
G=\F G\E G$ the generalized Fitting subgroup of $G$. If is
well-known~\cite[Chapter~$11$]{Aschbacher} that $\E G=E_1\cdots
E_\ell$ is a central product of 
$E_1,\ldots,E_\ell$ (here $\E{G}=1$ if $G$ has no components), that
$[\F G,\E G]=1$, that $\E{\E G}=\E G$ and 
that $\cent G{\FF G}\leq \FF G$. 

\begin{lemma}\label{id1}
Suppose that $K=E_1\cdots E_\ell$ is a central product of $\ell$
quasisimple groups with $E_i/\Z {E_i}\cong E_j/\Z{E_j}$, for each
$i,j\in \{1,\ldots,\ell\}$. Then $\mu(K)\leq 2$. 
\end{lemma}

\begin{proof}
As $E_i$ is quasisimple for each $i\in \{1,\ldots,\ell\}$, we have
$K=[K,K]$ and $K/\Z K=T_1\times \cdots \times T_\ell$ with $T_i=E_i\Z
K/\Z K$. By hypothesis there is a nonabelian simple group $T$ such
that, for each $i\in \{1,\ldots,\ell\}$, we have $T_i\cong
T$. We argue by
contradiction and we assume that $\mu(K)\geq 3$, that is, $K$ has
three proper subgroups $A_1$, $A_2$ and $A_3$ with $|K:A_1|$,
$|K:A_2|$ and $|K:A_3|$ relatively prime. Assume first that $A_i\Z
K< 
K$,  for each $j\in \{1,2,3\}$. So, $A_1 \Z K/\Z K$, $A_2 \Z K/\Z K$ and
$A_3 \Z K/\Z K$ are three proper subgroups of $K/\Z K$ with relatively
prime indices, that is, $\mu(K/\Z K)\geq 3$. Now,
from~\cite[Lemma~$5.2$]{DGPS} we have $\mu(T^\ell)\leq 2$, and hence we obtain a contradiction. This shows that $K=A_{j_0}\Z
K$, for some $j_0\in \{1,2,3\}$. Now, we have $K=[K,K]=[A_{j_0}\Z
K,A_{j_0}\Z K]=[A_{j_0},A_{j_0}]\leq A_{j_0}$, but this
contradicts the fact that $A_{j_0}$ is a proper subgroup of $K$.  
\end{proof}

\begin{lemma}\label{id2}
Let $G$ be a transitive permutation group on $\Omega$ and let $\omega\in\Omega$. Suppose 
that $N$ is normal in $G_\omega$ and $N$ fixes a unique point on
$\Omega$. Then the maximal size of a subset of pairwise coprime non-trivial subdegrees of $G$ is at most
$\mu(N)$. 
\end{lemma}

\begin{proof}
Let $\omega_1,\ldots,\omega_r$ be elements of
$\Omega\setminus\{\omega\}$ with $|\omega_i^{G_\omega}|$ relatively
prime to $|\omega_j^{G_\omega}|$, for distinct elements $i$ and
$j$ in $\{1,\ldots,r\}$. For every $i\in \{1,\ldots,r\}$, set
$N_i=N_{\omega_i}$. Since $N$ fixes only the point $\omega$ of
$\Omega$, the group $N_i$ is a proper subgroup of $N$. Furthermore, as
$N\unlhd G_\omega$, the index $|N:N_i|$ divides
$|\omega_i^{G_\omega}|$. From Definition~\ref{def}, we obtain $r\leq\mu(N)$.
\end{proof}

\begin{lemma}\label{lemmaemb0}
Let $H$ be a finite group such that $\cent H{\E H}=1$ and $\E H\cong
T^\ell$ for some nonabelian simple group $T$ and for some $\ell\geq
1$. Let $F$ be a subgroup of $H$ with $F=F_1\cdots F_{\ell'}$ the
central product of $\ell'$ quasisimple groups with $F_i/\Z {F_i}\cong
T$ and $\ell'\geq \ell$.
Then $F=\E H$.
\end{lemma}

\begin{proof}
Write $\E {H}=T_1\times \cdots \times T_\ell\cong T^\ell$ with
$T_i\cong T$ for each $i\in \{1,\ldots,\ell\}$. As $\cent H {\E H}=1$,
the group $H$ is isomorphic to a subgroup of $\Aut(\E H)$. So,
replacing $H$ by $\Aut(\E H)$ if necessary, we may assume that
$H=\Aut(\E H)$.  

Write $E=\E H$. We argue by induction on $\ell$. Assume that
$\ell=1$. Since $\Out(T)$ 
is soluble and $F=[F,F]$, we obtain $F\leq E $ and so $F=E$. Assume
that $\ell >1$. We prove a preliminary claim from which the 
proof will follow.

\smallskip

\noindent\textsc{Claim~1. }If $F$ is a subgroup of $\Sym(m)$, then $m\geq
\ell' d$, where $d$ is the minimal degree of a faithful permutation
representation of $T$. In particular, $m\geq 5\ell'$.

\smallskip

\noindent We prove it by induction on $|F|$.  Let
$\Lambda_1,\ldots,\Lambda_k$ be the orbits of $F$ on
$\{1,\ldots,m\}$ and let  $L_j$ be the
permutation group induced by $F$ on 
$\Lambda_j$. In particular, $L_j/\Z{L_j}\cong T^{\ell_j}$ for some $0\leq
\ell_j\leq \ell'$, and $\ell'\leq \sum_{j=1}^k\ell_j$. If, for each
$j\in \{1,\ldots,k\}$, we have 
$|L_j|<|F|$, then by induction we obtain $|\Lambda_j|\geq \ell_j
d$. In particular, $m=\sum_{j=1}^k|\Lambda_j|\geq
\sum_{j=1}^{k}\ell_j d\geq \ell' d$. Therefore, we may assume that
$|F|=|L_j|$ for some $j\in \{1,\ldots,\ell\}$, that is, $F$ acts
faithfully and transitively on $\Lambda_j$. In particular, replacing the set
$\{1,\ldots,m\}$ by $\Lambda_j$ if necessary, we may assume that $F$ is
a transitive subgroup of $\Sym(m)$.

Let $\mathcal{B}$ be the system of imprimitivity
consisting of the 
orbits of $\Z F$. Let $K$ be the kernel of the action of $F$ on
$\mathcal{B}$ and let $F^\mathcal{B}$ be the permutation group induced by
$F$ on $\mathcal{B}$. Clearly,  $\Z F\leq K$. Assume that $\Z F<
K$. In particular, since $F/\Z F\cong T^{\ell'}$, there exists $i\in
\{1,\ldots,\ell'\}$ with $F_i\leq 
K$. Let $B$ be a $\Z F$-orbit and $\lambda\in B$. Since $\Z F$ is
abelian, $\Z F$ acts 
regularly on $B$ and hence $F_i=\Z F (F_i)_\lambda$. In particular,
$F_i=[F_i,F_i]=[(F_i)_\lambda,(F_i)_\lambda]=(F_i)_\lambda$ and $F_i$
fixes the point $\lambda$ of $B$. Since $F_i\unlhd F$, $F$ is
transitive and $F_j\leq F_\lambda$, we see that $F_i=1$, a
contradiction. Thus $K=\Z F$ and 
$F^{\mathcal{B}}\cong T^{\ell'}$. From~\cite[Theorem~$3.1$]{Pr1}, we have
$|\mathcal{B}|\geq \ell' d$. Therefore, $m\geq |\mathcal{B}|\geq \ell' d$.
Finally, since $\Sym(4)$ is soluble, we have $d\geq 5$.~$_\blacksquare$  
 
\smallskip

Let $K$ be the kernel of the action by conjugation of $H$ on the
set 
$\{T_1,\ldots, T_\ell\}$ of $\ell$ simple direct factors of
$E$. Clearly, $F\cap K$ is a normal subgroup of $F$. Assume that
$F\cap K\leq \Z F$. Then $FK/K\cong F/(F\cap K)$ is 
isomorphic to a subgroup of $\Sym(\ell)$ and 
hence, by Claim~1 applied to  $F/(F\cap K)$, we obtain $\ell\geq
5\ell'$, a contradiction since  by assumption $\ell'\geq \ell$. Thus $F\cap K\nsubseteq \Z{F}$. Since $F\cap
K\unlhd F$ and $F/\Z F\cong T^{\ell'}$, there 
exists $i\in \{1,\ldots,\ell'\}$ with $F_i\leq  K$. Relabelling $F_i$ by
$F_1$ if necessary, we may assume 
that $F_1\leq K$.    

Since $K/E\cong \Out(T)^\ell$, $\Out(T)$ is soluble and
$F_1=[F_1,F_1]$, we see that $F_1\leq E$.  

For each $j\in \{1,\ldots,\ell\}$, let $\pi_j:E \to T_j$ be the
projection onto the $j^\textrm{th}$ coordinate of $E$ and let $L_j$ be
the kernel of $\pi_j$. Since $F_1\leq E $ and $L_j\unlhd E
$, we have $F_1\cap L_j\unlhd F_1$  and so either
$F_1\cap L_j\leq \Z {F_1}$ or $F_1\leq L_j$. Write $J=\{j\in
\{1,\ldots,\ell\}\mid F_1\leq L_j\}$. If $J=\{1,\ldots,\ell\}$, then
$F_1\leq \cap_{j=1}^\ell L_j=1$,
a contradiction. Thus, relabelling the set $\{1,\ldots,\ell\}$ if
necessary, we may assume that $J=\{m+1,\ldots,\ell\}$ for some $m\geq
1$. Fix  $j$ in $\{1,\ldots,m\}$. Now, as $F_1\cap L_j\leq \Z{F_1}$, we have $|T|\geq
|\pi_j(F_1)|=|F_1:F_1\cap L_j|=|F_1:\Z{F_1}||\Z{F_1}:F_1\cap L_j|=|T||\Z{F_1}:F_1\cap L_j|$
and hence $\pi_j(F_1)=T_j$ and $F_1\cap
L_j=\Z{F_1}$. Since this argument does not depend on $j\in \{1,\ldots,m\}$,
we have $\Z {F_1}=F_1\cap (\cap_{j=1}^mL_j)=F_1\cap (T_{m+1}\times
\cdots \times T_\ell)$. Moreover, since for each $j\in
\{1,\ldots,m\}$ we have $T_j=\pi_j(F_1)$, we see that $F_1\leq
D\times T_{m+1}\times \cdots 
\times T_\ell$ where
$D$ is a diagonal subgroup of
$T_1\times \cdots \times T_m$, that is, $D$ is conjugate under an
element of $H$ to the diagonal subgroup $\{(t,\ldots,t)\mid t\in T\}$ of $T_1\times \cdots \times T_m$. Summing up, this
gives $F_1=D\times  
\Z{F_1}$. As $F_1=[F_1,F_1]$, we have $\Z {F_1}=1$
and $F_1=D$.

Since $H=\Aut(E)$, we have $\cent H {F_1}\cong \Sym(m)\times
\Aut(T^{\ell-m})$. Let $A$ be 
the normal subgroup of $\cent H{F_1}$ isomorphic to $\Sym(m)$ and let
$B$ be the normal subgroup of $\cent H{F_1}$ isomorphic to
$\Aut(T^{\ell-m})$. Now the group $F_2\cdots F_{\ell'}$ is contained in $\cent H
{F_1}=A\times
B$. From Claim~1, $A$ contains at most $m/5$ of the components $F_2\cdots
F_{\ell'}$. Also, by induction, we have that $B$ contains
at most $\ell-m$ of the components $F_2\cdots F_{\ell'}$ and, if equality is met
then $F_2\cdots F_{\ell'}=T_{m+1}\times \cdots\times T_{\ell}$. Therefore,
$\ell'-1\leq m/5+\ell-m$. Since $\ell'\geq \ell$, this gives
$\ell'=\ell$, $m=1$, $F_1=T_1$ and $F_2\cdots
F_{\ell'}=T_2\times\cdots\times T_{\ell}$. In particular, $F=E$.  
\end{proof}

\begin{lemma}\label{lemmaemb1}
Let $H$ be a finite group and $E = \E H$.
Assume that $\cent {H}{E}$ is
  soluble and $E /\Z {E}\cong T^\ell$ for some nonabelian
  simple  
  group $T$ and for some $\ell\geq 1$. If $f:E \to H$ is an
  injective homomorphism, then $f(E )=E$.
\end{lemma}

\begin{proof}
We write $E= \E H$, $Z=\Z E$ and $\overline{H}=H/\cent H{E}$. Let
$-:H\to\overline{H}$ be the 
natural projection. Here we use the ``bar'' notation, that is, we
denote by $\overline{X}$ the image under $-$ of the subgroup $X$ of $H$.

In this paragraph we show that $\cent {\overline{H}}{\overline{E}}=1$. We have $\cent
{\overline{H}}{\overline{E}}=C/\cent H{E}$ for 
some subgroup $C$ of $H$. Since $[\overline{C},\overline{E}]=1$ and
$E\unlhd H$, we obtain $[C,E]\leq \cent H E\cap E=\Z E$. In
particular, $[[C,E],E]=1$ and $[[E,C],E]=1$. Now, from the Three Subgroup
Lemma, we have $[E,C]=[[E,E],C]=1$. Thus $C\leq \cent H E$ and
$\overline{C}=1$.

Since every component of $\overline{H}$ is either contained in
$\overline{E}$ or commutes with $\overline{E}$, and since $\cent {\overline{H}}{\overline{E}}=1$, we obtain that $\overline{E}=\E {\overline{H}}$.
 
Write $F=f(E)$ and $\overline{F}=\overline{f(E)}$. As $\cent H E$ is
soluble and $f$ is injective, we have $F\cap \cent H E\leq \Z{F}$ and
$\overline{F}\cong F/(F\cap \cent H E)$ is a central product of
$\ell$ quasisimple groups. Since $\overline{E}\cong
\overline{F}/\Z{\overline{F}}\cong T^\ell$, from
Lemma~\ref{lemmaemb0}, we 
have $\overline{F}=\overline{E}$. Therefore, $F\cent H E=E\cent H
E$. Since $F=[F,F]$, $E=[E,E]$ and $\cent H E$ is soluble, we obtain
that the last term of the derived series of $F\cent H E$ (respectively
$E\cent H E $) is $F$ (respectively $E$), that is, $F=E$. 
\end{proof}

\begin{proposition}\label{emb1}Let $G$ be a transitive permutation group on
  $\Omega$. For
  $\omega\in\Omega$, assume that $\cent{G_\omega}{\E{G_\omega}}$ is
  soluble, that $\E{G_\omega}/\Z{\E{G_\omega}}\cong T^\ell$ for some
  nonabelian simple group $T$ and for some $\ell\geq 1$, and that
  $G_\omega=\N G{\E {G_\omega}}$.  Then 
  the maximal size of a subset of  non-trivial pairwise coprime subdegrees of
  $G$ is at most $2$.   
\end{proposition}

\begin{proof}  
Fix $\omega\in \Omega$ and write $E=\E {G_\omega}$. Assume that $E$
fixes an element $\omega'$ of $\Omega$. Let $g\in G$ with
$\omega^g=\omega'$. Now $E^{g^{-1}}\leq G_\omega$ and so, from
Lemma~\ref{lemmaemb1} applied with $H=G_\omega$, we have
$E^{g^{-1}}=E$  and hence $g\in
\N G E=G_\omega$. This yields $\omega'=\omega$ and hence $E$ fixes a unique
point of $\Omega$. Now the proof follows from Lemmas~\ref{id1} and~\ref{id2}.
\end{proof}

Let $G=T_1\times\cdots \times T_\ell$ be the direct product of
nonabelian simple groups. We say that $T_i$ has {\em multiplicity} $r$
in $G$, if $G$ has exactly $r$ simple direct factors isomorphic to $T_i$.
\begin{lemma}\label{lemmaemb2}
Let $H$ be a finite group. Assume that each simple direct factor of $\E
H/\Z H$ has multiplicity at most $4$ and that $H$ has a unique
component $Q$ such that  $Q/\Z{Q}$ has largest order among the components of $H$. If
$f:Q\to H$ is an injective homomorphism, then $f(Q)=Q$.
\end{lemma}

\begin{proof}
Write $R=f(Q)$ and $\E H=E_1\cdots E_\ell$ with $E_1,\ldots,E_\ell$
the components of $\E H$. Set $\E H/\Z {\E H}=T_1\times \cdots \times
T_\ell$ with $\ell\geq 1$ and with $T_i$ a nonabelian simple group,
for each $i\in \{1,\ldots,\ell\}$. Relabelling the index set
$\{1,\ldots,\ell\}$ if necessary, we may assume that $E_1=Q$. The
group $H$ acts by conjugation on the set $\{T_1,\ldots,T_\ell\}$ of
$\ell$ simple direct factors of $\E H/\Z{\E H}$. The kernel of the
action of $H$ on $\{T_1,\ldots,T_\ell\}$ is $K=\cap_{i=1}^\ell\N H
{T_i}$. Since $T_i$ has multiplicity at most $4$ in $\E{H}/\Z{\E H}$,
we see that $H/K$ has orbits of length at most $4$ and hence $H/K$ is
soluble.

As $R$ is quasisimple, this yields $R\leq K$. As $\Out(T_i)$ is soluble for each $i\in
\{1,\ldots,\ell\}$ and since $K/\E H$ is isomorphic to a subgroup of
$\Out(T_1)\times \cdots \times \Out(T_\ell)$, we obtain that $K/\E H$
is soluble. As $R$ is quasisimple, we get $R\leq \E H$.    

For each $j\in \{1,\ldots,\ell\}$, let $\pi_j:\E H\to T_j$ the natural
projection onto the $j^{\mathrm{th}}$ factor $T_j$ of $\E H/\Z{\E H}$
and let $L_j$ be the kernel of $\pi_j$. Since $R\cap L_j\unlhd R$, we
have that either $R\cap L_j\leq \Z R$ or $R\leq L_j$. In the former
case, $|T_j|\geq \pi_j(R)=|R:R\cap L_j|\geq |R:\Z R|=|T_1|$ and hence
$j=1$ because of the maximality and uniqueness of $|T_1|$. Therefore this yields
$R\leq \cap_{j=2}^\ell L_j=E_1\Z{\E 
  H}$. Since $R=[R,R]$, we obtain $R\leq [E_1\Z{\E H},E_1\Z{\E
    H}]=E_1$ and hence, since $f$ is injective, $R=E_1$.    
\end{proof}

\begin{proposition}\label{emb2}
Let $G$ be a transitive permutation group on
  $\Omega$. For
  $\omega\in\Omega$, assume that each simple direct factor of
$\E{G_\omega}/\Z{\E {G_\omega}}$ has multiplicity at most $4$, and that
$\E{G_\omega}$ has a unique component $Q$ such that  $Q/\Z {Q}$ has largest order among the
components of $\E H$. Suppose that $\N G Q =G_\omega$.  Then 
  the maximal size of a subset of  non-trivial pairwise coprime subdegrees of
  $G$ is at most $2$.
\end{proposition}

\begin{proof}  
If $Q$ fixes the element $\omega'$ of $\Omega$, then there exists
$g\in G$ with $\omega'=\omega^g$ and $Q^{g^{-1}}\leq G_\omega$. From
Lemma~\ref{lemmaemb2}, we have $Q^{g^{-1}}=Q$  and so $g\in
\N G Q =G_\omega$. This yields $\omega'=\omega$ and $Q$ fixes a unique
point of $\Omega$. Now the proof follows from Lemmas~\ref{id1} and~\ref{id2}. 
\end{proof}

The following proposition is taken from~\cite[Theorem~$3.7$]{Wielandt}.
\begin{proposition} \label{sylow}  Let $G$ be a transitive permutation
  group on $\Omega$. For $\omega\in \Omega$, assume that $G_\omega$
  contains the normalizer  of a Sylow $p$-subgroup of $G$.  Then $p$
  divides the degree of every non-trivial suborbit of $G$. 
\end{proposition} 

\begin{proof}See~\cite[Theorem~$3.7$]{Wielandt}.
\end{proof}

\section{Auxiliary lemmas}\label{aux}

We say that a factorization $H=AB$ is
\emph{coprime} if $|H:A|$ is relatively prime to $|H:B|$ and both
$A,B$ are proper subgroups of $H$ (see~\cite[Section~$2$]{DGPS}). Also
$H=AB$ is \emph{maximal} if $A$ and $B$ are maximal subgroups of $H$.

\begin{lemma}\label{lucky}Let $H$ be a finite group, $r$ a prime, and $R$ a normal
  $r$-subgroup of $H$. Assume that $H/R=E_1\cdots E_\ell$ is a central
  product of $\ell$ quasisimple groups with $E_i/\Z{E_i}\cong
  E_j/\Z{E_j}$, for each $i,j\in \{1,\ldots,\ell\}$. Then $\mu(H)\leq
  3$. 

Assume further that $\mu(H)=3$ and let $A_1,A_2$ and $A_3$ be maximal
subgroups of $H$ 
  with $|H:A_i|$ relatively prime to $|H:A_j|$, for distinct $i$ and
  $j$ in $\{1,2,3\}$. Let $U$ be the normal subgroup of $H$ with
  $U/R=\Z{H/R}$. Then relabelling the set $\{A_1,A_2,A_3\}$ if
  necessary, $|H:A_3|$ is divisible by $r$, $U\leq A_1,A_2$ and
  $H/U=(A_1/U)(A_2/U)$ is a coprime maximal factorization of $H/U$.
\end{lemma} 

\begin{proof}
Suppose that $\mu(H)\geq 3$ and let $A_1,A_2$ and $A_3$ be maximal
subgroups of $H$  with $|H:A_i|$ relatively prime to $|H:A_j|$, for
distinct $i$ and 
  $j$ in $\{1,2,3\}$.
Since $A_i$ is maximal, we obtain that 
either $H = A_i R$ or $R \leq A_i$. In the former case, the index $|H
: A_i | = |R : A_i \cap R|$ is 
divisible by $r$. In the latter case, $A_i /R$ is a maximal subgroup
of $H/R$. From Lemma~\ref{id1}, we have $\mu(H/R) \leq 2$ and so, in
particular, $\mu(H)=3$ (since we are assuming $\mu(H)\geq 3$). We note
that this proves the first assertion of the lemma. Since $\mu(H/R)\leq
2$, there exists exactly one 
element $A_i$ in $\{A_1 , A_2 , A_3 \}$ with $R\nsubseteq A_i$ , and
there are exactly two elements $A_j$ and $A_k$ 
in $\{A_1 , A_2 , A_3 \}$ with $R \leq A_j , A_k$. Thus, replacing
$A_i$ by $A_3$ if necessary, we may 
assume that $R \leq A_1 , A_2$ and that $|H : A_3 |$ is divisible by
$r$. Since $A_1 /R$ and $A_2 /R$ 
are maximal subgroups of $H/R$ and as $H/R$ is a central product of
quasisimple groups, we have that $U \leq A_1 , A_2$.
Hence $H/U = (A_1 /U )(A_2 /U )$ is a maximal factorization of
the characteristically
simple group $H/U$ with $\gcd(|H : A_1 |, |H : A_2 |) = 1$.
\end{proof}

%\begin{lemma}\label{evenAlt}
%Let $\Alt(n)=AB$ be a maximal factorization of $\Alt(n)$ with
%$|\Alt(n):A|$ relatively prime to $|\Alt(n):B|$ and with $n\geq
%5$. Then either $n=7$, or one of $|\Alt(n):A|$ or $|\Alt(n):B|$ is
%even. 
%\end{lemma}

%\begin{proof}
%The triple $(\Alt(n),A,B)$ is in~\cite[Table~1]{DGPS}. Assume that
%$|\Alt(n):A|$ and $|\Alt(n):B|$ are both odd. From~\cite[Table~$1$]{DGPS}, we s%ee that 
%$n=p^\ell$ for some prime $p$ 
%and $\ell\geq 1$. If $n\leq 8$, then the result is true by a direct
%inspection. Assume that $n\geq  9$. From~\cite[Table~$1$]{DGPS}, we see
%that, replacing $A$ by $B$ if necessary, $A$ is the stabilizer of a
%point. In particular, $|\Alt(n):A|=p^\ell$ and so $p$ is an odd prime. Suppose
%that $\ell=1$. From~\cite[Table~$1$]{DGPS}, the group $B$ is
%primitive. Since $B$ contains a Sylow $2$-subgroup of $\Alt(n)$, we
%obtain that $B$ contains an involution with support of size $4$. It follows
%from~\cite[Example~$3.3.1$]{DM} that $n\leq 8$. 
%
%Suppose that $\ell\geq 2$. From~\cite[Table~$1$]{DGPS}, the group $B$ is a
%maximal imprimitive subgroup of $\Alt(n)$ and so
%$|\Alt(n):B|=p^\ell!/((p^k!)^{p^{\ell-k}}p^{\ell-k}!)$. Now it is a
%tedious calculation to check that $|\Alt(n):B|$ is even.
%\end{proof}

Given a finite group $G$, we say that the normal subgroup $N$ of $G$
is the {\em last term of the derived series} of $G$ if $G/N$ is
soluble and $N=[N,N]$.

\begin{lemma}\label{derivedSeries}
Let $T$ be a transitive permutation group on $\Omega$, let $\omega$
be in $\Omega$ and let $N$ be the last term
of the derived series of $T_\omega$. If $T_\omega=\N T N$, then $N$ fixes only the point $\omega$. 
\end{lemma}
\begin{proof}
Suppose that $N$ fixes $\omega'$ and write
$\omega'=\omega^g$, for some $g\in T$.  Set
$K=N^{g^{-1}}$. Since $K\leq T_\omega$, $T_\omega/N$ is soluble and
$NK/N$ is isomorphic 
to $K/(K\cap N)$, we see that $K/(K\cap N)$ is soluble. Since
$K=[K,K]$, we obtain that $K=N\cap K$ and so $N=K=N^{g^{-1}}$. This shows
that $g\in \N T N=T_\omega$. So $\omega'=\omega$ and  $N$ fixes only
the point $\omega$ of $\Omega$.
\end{proof}

\begin{remark}
\label{ps-max}
{\rm Let $T$ be a nonabelian simple permutation group on a set $\Omega$ and 
let $T_\omega$  be  pseudo-maximal in $T$, with $\omega \in \Omega$.
So, there exists an almost simple group $A$ with socle $T$ and a maximal subgroup $M$ of $A$ such that
$T\nsubseteq M$ and $T_\omega=T\cap M$. 
Let  $N$ be  a characteristic subgroup of $T_\omega$.  Then $M=\N A N$, because $M$ is maximal in $A$
and $T_\omega\unlhd M$. 
Hence, $T_\omega=T\cap \N AN=\N T N$. Furthermore, as $T_\omega=\N T {T_\omega}$, we obtain that
$\omega$ is the only fixed point of $T_\omega$ in $\Omega$. We will use these two facts
repeatedly in the following. }
\end{remark}

%\begin{lemma}\label{class2}
%Let $N$ be a group with a normal subgroup $R$ with $N/R\cong \Alt(n)$,
%for some $n\geq 5$. If $R$ is a $2$-group, then either $\mu(N)\leq 2$
%or $n=7$.
%\end{lemma}

%\begin{proof}
%Suppose that $\mu(N)\geq 3$ and let $A_1,A_2$ and $A_3$ be maximal
%subgroups of $N$  with $|H:A_i|$ relatively prime to $|H:A_j|$, for
%distinct $i$ and   $j$ in $\{1,2,3\}$. From Lemma~\ref{lucky}
%and~\ref{evenAlt} we obtain $n=7$. 
%\end{proof}

%\begin{lemma}\label{Alt7}
%Let $T$ be a transitive permutation group on $\Omega$, let $\omega$
%be in $\Omega$, let $C$ be a cyclic group of even order and let $N$ be the
%last term of the derived series of $T$. Assume that $N\cong C^6\rtimes
%\Alt(7)$ and let $E$ be the elementary abelian $2$-group of order $64$
%contained in the Fitting subgroup of $N$. If $T_\omega=\N T E$, then every %non-trivial subdegree of
%$T_\omega$ is even. 
%\end{lemma}
%\begin{proof}
%Since $\N T N\leq \N T E=T_\omega$, from Lemma~\ref{derivedSeries}, $N$ fixes %only the point $\omega$ of
%$\Omega$. Let $x$ be an
%involution of $\Alt(7)$. Since $\cent E x$ has order $16$, we see that
%$N$ contains a unique elementary abelian $2$-group of order $64$. 

%We show that $E$ fixes only the point $\omega$ of $\Omega$. If $V$
%fixes $\omega^g$, for some $g\in T$, then $E^{g^{-1}}\leq
%T_\omega$. From the previous paragraph, $E=E^{g^{-1}}$, $g\in \N T
%E\leq T_\omega$ and $E$ fixes only the point $\omega$. This yields
%that every non-trivial subdegree of $T$ is even.
%\end{proof}

\section{Alternating groups}\label{alt}

\begin{proof}[Proof of Theorem~A for the alternating groups]
A subgroup $H$ of $ \Sym(n)$ is either intransitive, imprimitive or primitive in its action
on \{1,\ldots,n\}. In the proof of this theorem we consider these three
cases separately.

Let $T=\Alt(n)$,  for some $n\geq 5$.
Fix $\omega\in\Omega$ and write $H=T_\omega$. Assume that  $H$ is 
intransitive in the natural action of $T$ of degree $n$. Then $H\cong
(\Sym(k)\times \Sym(n-k))\cap T$, for some $k$ with $1\leq k<n/2$. (Note that for $n$ even, $(\Sym(n/2)\times \Sym(n/2))\cap T$ is not pseudo-maximal in $T$.) In
particular, the 
action  of  $T$ on $\Omega$ is permutation equivalent to the action of 
$\Alt(n)$ 
on the $k$-subsets of $\{1,\ldots,n\}$. Suppose that
$n-k\geq 5$. Let $N$ be the minimal normal
subgroup of $H$ isomorphic to $\Alt(n-k)$. Clearly, $N$ is simple and
fixes a unique 
$k$-subset of $\{1,\ldots,n\}$. So, by Lemmas~\ref{id1} and~\ref{id2},
the group $T$ has at 
most $2$ non-trivial coprime subdegrees on $\Omega$. Now, suppose that
$n-k\leq 4$. If $k\leq 2$, then the rank of $T$ is at most $3$ and the
assertion follows immediately. If $k\geq 3$, then $(n,k)=(7,3)$ and by direct
inspection we see that $T$ has no pair of non-trivial coprime subdegrees.

Assume next that $H$ is imprimitive in the natural action of $T$ of
degree $n$. Then 
$H\cong (\Sym(k)\wr \Sym(n/k))\cap T$, for some divisor $k$ of $n$ with
$1<k<n$. In particular, the action of $T$ on $\Omega$ is permutation equivalent
to the action of $\Alt(n)$ on the set $\mathcal{P}$ of partitions of
$\{1,\ldots,n\}$ into $n/k$ parts all of size $k$. Suppose that $k\geq
5$. Let $N$ be the socle of $H$. Clearly, $N\cong \Alt(k)^{n/k}$ and
$N$ fixes a unique element of $\mathcal{P}$. So,
Lemmas~\ref{id1} and~\ref{id2} yield that $T$ has at most $2$
non-trivial pairwise coprime subdegrees on $\Omega$. It remains to consider the
case that $k\in \{2,3,4\}$. Let $N$ be the normal subgroup of $H$ isomorphic
to $\Sym(k)^{n/k}\cap T$. Clearly, $N$ fixes a unique element of $\mathcal{P}$. Furthermore, since
$N$ is a $\{2,3\}$-group, we have $\mu(N)\leq 2$. Therefore
Lemma~\ref{id2} yields that $T$ has at most $2$ non-trivial coprime subdegrees.

It remains to consider the case that $H$ is a primitive subgroup of
$T$ in the natural action of degree $n$. Let $N$ be the socle of
$H$. Suppose that $N\cong S^\ell$ for 
some nonabelian simple group $S$ and $\ell\geq 1$. Clearly, $N=\E H$
and $\cent H N=1$, and $\N T N=H$ because $H$ is pseudo-maximal in
$T$. In particular, from Proposition~\ref{emb1} we see that $T$ has at
most $2$ non-trivial pairwise coprime subdegrees. Finally assume that $N$ is an
elementary abelian $p$-group. In the rest of the proof, we identify
$\mathrm{AGL}(d,p)$ with its image under the natural  affine
permutation representation. So $H=\textrm{AGL}(d,p)\cap T$ and hence $H$ is isomorphic to a subgroup of
index $2$ in $\mathrm{AGL}(d,p)$ if $p$ is odd, and $H=\mathrm{AGL}(d,p)$ if $p=2$. (Note that the affine general linear group $\mathrm{AGL}(d,p)$ is a subgroup of $\Alt(p^d)$ only for $p=2$.)  If $d=1$, then by Proposition~\ref{sylow}
every subdegree of $T$ is divisible by $p$. Assume now that $d>1$. Suppose
that $T$ has two coprime subdegrees $n_1=|\omega_1^{H}|$ and
$n_2=|\omega_2^{H}|$. We show that either $n_1$ or $n_2$ is divisible by
$p$, from which it follows that $T$ has at most $2$ non-trivial
coprime subdegrees. We argue by contradiction and we assume that $n_1$
and $n_2$ are not divisible by $p$. In particular, each of $H_{\omega_1}$ and
$H_{\omega_2}$ contains a Sylow $p$-subgroup of $H$. Therefore,
from~\cite[Theorem~$1$]{Seitz} we have that  $H_{\omega_1}=
(N\rtimes P_1)\cap T$ and $H_{\omega_2}=(N\rtimes P_2)\cap T$ with
$P_1$ and $P_2$  maximal parabolic subgroups of $\mathrm{GL}(d,p)$, that is, 
\[
P_i\cong\left\{
\left(
\begin{array}{cc}
A&B\\
0&C\\
\end{array}
\right)\mid A\in \mathrm{GL}(d_i,p),\mathrm{GL}(d-d_i,p), B\in
\mathrm{Mat}(d_i\times (d-d_i),p) 
\right\}
\]
where $1\leq d_i\leq d-1$, for $i=1,2$. For each $i\in\{1,2\}$, we
have $N\leq H_{\omega_i}$, and so, from the modular law, we obtain
$H_{\omega_i}=N\rtimes (P_i\cap T)$. Therefore

$$n_i=|H:H_{\omega_i}|=|\GL(d,p)\cap T: P_i\cap T|=|\GL(d,p):P_i|,$$
for $i=1,2$. Since $n_1$ and $n_2$ are coprime, $\GL(d,p)=P_1P_2$,
leading to a factorization 
of $\mathrm{PGL}(d,p)$ by two maximal parabolics. No such
factorization exists, see for example~\cite[Table~$1$]{LPS}.
\end{proof}

\begin{table}[!h]
\begin{tabular}{|l|l|l|}\hline
Group& Soluble case&Isomorphisms\\\hline
$\PSL_n(q)$&$n=1$ or $(n,q)=(2,2),(2,3)$&\\
$\PSp_n(q)$&$(n,q)=(2,2),(2,3)$&$\PSp_2(q)=\PSL_2(q)$\\
$\PSU_n(q)$&$n=1$ or $(n,q)=(2,2),(2,3),(3,2)$&$\PSU_2(q)=\PSL_2(q)$\\
$\POmega_n(q)$&$n=1$ or $(n,q)=(3,3)$&$\POmega_3(q)=\PSL_2(q)$,
$\POmega_5(q)=\PSp_4(q)$\\
$\POmega_n^+(q)$&$n=2$ or
$(n,q)=(4,2),(4,3)$&$\POmega_4^+(q)=\PSL_2(q)\times \PSL_2(q)$\\
&&$\POmega_6^+(q)=\PSL_4(q)$\\
$\POmega_n^-(q)$&$n=2$&$\POmega_4^-(q)=\PSL_2(q^2)$\\
&&$\POmega_6^-(q)=\PSU_4(q)$\\\hline
\end{tabular}
\caption{Some information on simple classical groups}
\label{table2}
\end{table}

\section{Classical groups}\label{classical}
In this section we prove Theorem~\ref{theorem} when the simple group
$T$ is a classical group. 
We use Aschbacher's theorem,  which subdivides  the maximal subgroups
of the almost simple groups 
with socle $T$ in nine classes $\mathcal{C}_1,\ldots,\mathcal{C}_8$ and
$\mathcal{S}$. In particular, in what follows we use the notation, the
treatment and
the terminology in~\cite[Chapter~$3$ and~$4$]{KL}.

We start with a preliminary proposition which will prove to be helful
in the proof of the main result of this section. First we set some
notation and some terminology. 

\begin{notation}\label{nota}{\rm 
Let $V$ be an $n$-dimensional vector
space over a field $\mathbb{F}_q$ of size $q$ and let $V_1\oplus \cdots \oplus V_{t}$ be a direct sum
decomposition of $V$ into $t$ subspaces. Let
$H$ be a subgroup $\GL(V)$ leaving invariant each summand of this
decomposition, that is, $V_i^h=V_i$ for all $h\in H$ and for all $i\in
\{1,\ldots,t\}$. Let
$H_i$ be the linear group induced by $H$ in its action on $V_i$. 
Note that $H_i$ centralizes $V_j$ (that is, $H_i$ acts trivially on $V_j$), for each $j\neq i$. We
assume that, for each $i\in \{1,\ldots,t\}$, the subspace $V_i$ is
an irreducible $H_i$-module. Fix $i$ and $j$ two distinct elements of
$\{1,\ldots,t\}$. We suppose that for each $a_i\in H_i$, there
exists $a_j\in H_j$ with $a_ia_j\in H$. (In particular, the element $a_ia_j$ of $H$ acts trivially on
$V_k$, for each $k\neq i,j$.)
Finally, we assume that for
each $i$, the group $H_i$ contains an element fixing no non-zero
vector of $V_i$.} 
\end{notation}

\begin{proposition}\label{claaa}Let $V$ and $V_1,\ldots,V_{t}$
  be as in Notation~\ref{nota}. If $t\geq 3$, then
  $V_1\oplus\cdots \oplus V_{t}$ is the unique decomposition of $V$ as a  direct sum
  of irreducible $H$-submodules of $V$. 
\end{proposition}
\begin{proof}
Assume that $t\geq 3$. Let $U$ be an irreducible $H$-submodule of $V$. We show that
$U=V_i$, for some $i\in \{1,\ldots,t\}$, from which the proof
follows. Let $u$ be a non-zero element of $U$ and write $u=u_1+\cdots+u_{t}$
with $u_i\in V_i$. Since $u\neq 0$, relabelling the direct summands
$\{V_i\}_i$ if necessary, we may assume that $u_1\neq 0$. Using
Notation~\ref{nota}, choose $a_1\in H_1$ fixing no non-zero element of
$V_1$. From
Notation~\ref{nota}, we see that there exists $a_2\in H_2$ with
$a=a_1a_2\in H$. Now, as $a$ centralizes $u_3,\ldots,u_t$, we obtain  $u-u^a=(u_1-u_1^{a_1})+(u_2-u_2^{a_2})\in U\cap
(V_1\oplus V_2)$. So, replacing $u$ by $u-u^a$ if necessary, we may
assume that $u=u_1+u_2\in V_1\oplus V_2$ and that $u_1\neq 0$. 

Since $t\geq 3$, from Notation~\ref{nota} we see that there exists
$a_3\in H_3$ with 
$b=a_1a_3\in H$. Then $u-u^b=u_1-u_1^{a_1}\in U\cap V_1$. So,
replacing $u$ by $u-u^b$ if necessary, we may 
assume that $u\in V_1$. Since $H$ acts irreducibly on $U$, we obtain 
$U=\langle u^h\mid h\in H\rangle\leq V_1$. As $V_1$ is an irreducible
$H$-module, we have $U=V_1$.    
\end{proof}

We observe that Proposition~\ref{claaa} does not hold for $t = 2$. Consider,
 for instance, the group \[
H=\left\{\left(
\begin{array}{cc}
a&0\\
0&a
\end{array}\right)\left|\right. a\in \mathbb{F}_q\setminus\{0\}
\right\}
\]
 of scalar matrices acting 
on the $2$-dimensional vector space $\mathbb{F}_q^2$. If $V_1=(1,0)\mathbb{F}_q$
 and $V_2=(0,1)\mathbb{F}_q$, then $V=V_1\oplus V_2$ is a direct decomposition that satisfies Notation~\ref{nota} with $t=2$ (here the group induced by $H$ on $V_i$ is the multiplicative group of the field $\mathbb{F}_q$ acting by multiplication). Clearly, every pair of $1$-dimensional subspaces of $V$ forms an $H$-invariant decomposition, and hence Proposition~\ref{claaa} does not hold for $t=2$.

\begin{notation}\label{nota1}{\rm 
Let $V$ be an $n$-dimensional vector space over a field $\mathbb{F}_q$
of size $q$. We let  $G$ be a subgroup of $\GL(V)$ and we suppose that
$G=\SL(V)$, or that $V$ is endowed with a  non-degenerate Hermitian,
symplectic or quadratic form and $G=\SU(V)$, $\Sp(V)$ or
$\Omega^\varepsilon(V)$ (with $\varepsilon\in \{\circ,+,-\}$)
respectively. Write $T=G/\Z{G}$ and assume that $T$ is a nonabelian
simple group. Assume that $T$ is a transitive permutation group on
$\Omega$ with pseudo-maximal point stabilizer $T_\omega$. 

Let $A$
be an almost simple group with socle $T$ and $M$ be a maximal subgroup of
$A$ with $T\nsubseteq M$ and with $T_\omega=M\cap T$. Suppose that $M$
lies in the Aschbacher class $\mathcal{C}_2$, that is,
$M$ is the stabilizer  in $A$ of a
direct sum decomposition $V_1\oplus\cdots\oplus V_{n/m}$ of $V$.
So, $M$ is of type
$\GL_m(q)\wr\Sym(n/m)$ if $T=\PSL_n(q)$, of type $\GU_m(q)\wr
\Sym(n/m)$ if $T=\PSU_n(q)$, of type $\Sp_m(q)\wr\Sym(n/m)$ if
$T=\PSp_n(q)$, and of type $\OO_m^\xi(q)\wr\Sym(n/m)$ if
$T=\POmega_n^{\varepsilon}(q)$  (see~\cite[Chapter~$3$ and~$4.2$]{KL}
for details and terminology). 
} 
\end{notation}

\begin{proposition}\label{claaa1}Let $T,\Omega$ and $M$ be as in
  Notation~\ref{nota1}. If
  $n/m\geq 3$ and if $(T,M)\neq (\PSL_n(2),\GL_1(2)\wr \Sym(n))$,
  $(\POmega_{n}^\varepsilon,\OO_2^+(2)\wr \Sym(n/2))$ or \\
  $(\POmega_n^\varepsilon(3), \OO_2^+(3)\wr \Sym(n/2))$, then the
  kernel of the $T_\omega$-action on the $V_i$ fixes a unique point of
  $\Omega$.  
\end{proposition}

\begin{proof}
Let $\overline{H}$ be the subgroup of $G$ leaving invariant each direct
summand $V_i$ of $V$, let $H$ be the projection of $\overline{H}$
in $T_\omega$ and, for each $i\in \{1,\ldots,n/m\}$, let
$\overline{H}_i$ be the matrix group induced by $\overline{H}$ in its
action on $V_i$. In particular, $H$ is the kernel of the
$T_\omega$-action on the $V_i$. Furthermore, we have $\overline{H}_i=\GL(V_i)$, $\GU(V_i)$,
$\Sp_m(V_i)$ and $\OO^\xi(V_i)$ respectively. Note that $H\unlhd
T_\omega$. From~\cite[Chapter~$2$]{KL}, we see that $\overline{H}_i$
acts irreducibly on $V_i$, except when $\overline{H}_i\cong\GL_1(q)$ and
$q=2$, or $\overline{H}_i\cong\OO_2^+(q)$ and $q=2,3$. Furthermore, for
each distinct $i$ and $j$, and for each $a_i\in\overline{H}_i$, there
exists $a_j\in\overline{H}_j$ with $a_ia_j\in
\overline{H}$. Finally, for each $i$, if $\overline{H}_i\not\cong
\GL_1(2)$, we see with a direct inspection that  $\overline{H}_i$ 
contains an element fixing no non-zero vector of $V_i$. This shows
that for $\overline{H}_i\not\cong \GL_1(2)$, $\OO_2^+(2)$ and
$\OO_2^{+}(3)$ we are in the position to apply Proposition~\ref{claaa}.

From 
%Notation~\ref{nota1} and 
Proposition~\ref{claaa}, the group
$\overline{H}$ fixes a unique direct sum decomposition of $V$ in $n/m$
vector spaces of dimension $m$. Assume that $H$ fixes $\omega'$ and write
 $\omega'=\omega^g$, for some $g\in T$. Let $\overline{g}\in G$ be an
 element projecting to $g$ in $T$. Now, $\overline{H}$ fixes the
 direct sum decomposition $V_1^{\overline{g}}\oplus \cdots\oplus
 V_{n/m}^{\overline{g}}$. From Proposition~\ref{claaa}, we obtain that
 $\overline{g}$ stabilizes the direct sum decomposition
 $V_1\oplus\cdots\oplus V_{n/m}$. So from the maximality of $M$ in
 $A$, we have that $g\in M\cap T=T_\omega$ and $\omega'=\omega$. 
\end{proof}

\begin{proof}[Proof of Theorem~A for the classical groups]
By the results in Section~\ref{alt}, we may assume that $T$ is one of: 
$\PSL_n(q)$ for $n\geq 2$ with $(n,q)\neq (4,2)$ and, if $n=2$, then
$q\geq 7$ and   $q\neq 9$; 
$\PSU_n(q)$ with $n\geq 3$ and $(n,q)\neq (3,2)$;  
$\PSp_n(q)$ with $n\geq 4$ and $(n,q)\neq (4,2)$; $\POmega_n(q)$ with
$n\geq 7$ and $nq$ odd;  or
$\POmega_n^{\pm}(q)$ with $n\geq 8$ and $n$ even.  Write $q=p^f$ for
some prime $p$ and some $f\geq 1$. We assume that $T$ is transitive on $\Omega$ and that, for $\omega\in\Omega$, $T_\omega=T\cap M$, where $M$ is a maximal subgroup of some almost simple group $A$ with socle $T$, and $T\not\subseteq M$.

In order to avoid a long list of exceptions in some general arguments that
we use later in the proof, we first deal with the case $T=\PSL_2(q)$
and we use Dickson's classification of the subgroup lattice of $T$
(see~\cite[Section~$3.6$, Theorem~$6.25$,~$6.26$]{Suzuki}). As above
$q\geq 7$ and $q\neq 9$. If $T_\omega\cong \Sym(3)$,
$\Alt(4)$, $\Sym(4)$ or $\Alt(5)$ (that is, $T_\omega$ is as
in~\cite[Theorem~$6.25$~(c)]{Suzuki}), then by direct inspection we see that
$\mu(T_\omega)\leq 2$ and the result follows from Lemma~\ref{id2} (applied with $N=T_\omega$). Assume that $M$ contains the stabilizer of a subfield of $\mathbb{F}_q$, that is, $T_\omega=M\cap T\cong \PSL(2,r)$ or
$\PGL(2,r)$ for $r=p^s$ with $s$ dividing $f$ (that is, $T_\omega$ is as
in~\cite[Theorem~$6.25$~(d)]{Suzuki}). If $r\neq 2$ or $3$, then from
Proposition~\ref{emb1} each set of pairwise coprime non-trivial subdegrees of $T$ has size at most two. If $r=2$ or $3$, then we have already
dealt with these cases as $\PSL_2(2)=\PGL_2(2)\cong
\Sym(3)$, $\PSL_2(3)\cong \Alt(4)$ and $\PGL_2(3)=\Sym(4)$.

Assume that  $M$ contains a parabolic subgroup, that is, $T_\omega$ is
a Borel subgroup of $T$ (here $T_\omega$ is as
in~\cite[Theorem~$6.25$~(a)]{Suzuki}). In particular, the action of $T$
on $\Omega$ is permutation equivalent to 
the action of $T$ on the projective line. Therefore $T$ is $2$-transitive
and has only one non-trivial subdegree, namely $q$.

Assume that $M$ contains the normalizer of a maximal torus of $T$,
that is, $T_\omega$ is a dihedral group of order $2(q\pm
1)/\gcd(2,q-1)$ (here $T_\omega$ 
is as in~\cite[Theorem~$6.25$~(b)]{Suzuki}). If $T_\omega$ is a $2$-group, then
every non-trivial subdegree of $T$ is even. Suppose that $T_\omega$ is
not a $2$-group and let 
$r$ be a prime with $r\mid |T_\omega|$ and $r\neq 2$. Let $R$ be a
Sylow $r$-subgroup of $T_\omega$. From the description of the subgroup
lattice of $T$ in~\cite[Theorem~$6.25$,~$6.26$]{Suzuki}, we see that
$R$ is a Sylow $r$-subgroup of $T$ and $\N T R\leq T_\omega$. In
particular, from 
Proposition~\ref{sylow}, every non-trivial subdegree of $T$ is
divisible by $r$. This concludes the analysis for $\PSL_2(q)$.

Now, to avoid a few more small exceptions in the general arguments below we
consider separately the cases where $T=\PSL_3(3)$, $\PSL_3(4)$,
$\PSL_4(3)$, $\PSU_3(3)$, $\PSU_4(3)$ and 
$\PSp_4(3)$. In each of 
these groups, we  see with a direct inspection with 
\texttt{magma}~\cite{magma} or with~\cite{ATLAS} that the theorem
holds true. Finally, for the remaining cases we use Aschbacher's
theorem and in particular we use extensively Tables~$3.5A$--$F$
in~\cite{KL}.  

\smallskip

\noindent\textsc{Case $M\in\mathcal{C}_1$: }$M$ is the stabilizer of 
totally singular or non-singular subspaces. 

\noindent We first consider the case that $M$ is of type
$P_m$, that is, $M$ is a maximal parabolic subgroup of $A$. In
particular, $M$ and, hence also $T_\omega$, contain the normalizer of a
Sylow $p$-subgroup of $T$. It
follows from Proposition~\ref{sylow} that every non-trivial subdegree
of $T$ is divisible by $p$. 

Now suppose that $M$ is of type
$\GL_m(q)\oplus\GL_{n-m}(q)$ if $T=\PSL_n(q)$, of type  $\GU_m(q)\perp
\GU_{n-m}(q)$ if $T=\PSU_n(q)$, of type $\Sp_m(q)\perp\Sp_{n-m}(q)$ if
$T=\PSp_n(q)$, of type $\OO_m(q)\perp\OO_{n-m}^{\varepsilon}(q)$ if
$T=\POmega_{n}(q)$, of type $\OO_m^\varepsilon(q)\perp
\OO_{n-m}^{\varepsilon}(q)$ if $T=\POmega_{n}^+(q)$, and of type
$\OO_{m}^\varepsilon\perp\OO_{n-m}^{-\varepsilon}(q)$ if
$T=\POmega_n^-(q)$. Note that, from~\cite[Table~$3.5A$--$F$]{KL}, we
take $m<n-m$ (except for $T=\POmega_n(q)$ and possibly for
$T=\POmega_{n}^-(q)$). Moreover, if 
$T=\POmega_{n}^-(q)$ and $n=2m$, then $m$ is even and $M$ is of type
$\OO_m^+(q)\perp\OO_m^{-}(q)$ with
$\POmega_m^{+}(q)\not\cong\POmega_{m}^{-}(q)$
(see~\cite[Proposition~$4.1.6$]{KL}). With a direct inspection in each
of these cases and using~Table~\ref{table2}, we see
that either $(i):$ each simple direct factor of $\E{T_\omega}/\Z{\E
  {T_\omega}}$ has multiplicity at most two and there exists a
unique factor having size strictly bigger than the others, or~$(ii):$
$\E{T_\omega}/\Z{\E{T_\omega}}$ is the direct product of pairwise
isomorphic simple groups, or~$(iii):$ $T_\omega$ is
soluble. Indeed,~$(iii)$ arises if and only if $T=\PSL_n(q)$ and
$(n,m,q)=(3,1,2),(3,1,3)$, or $T=\PSU_n(q)$ and
$(n,m,q)=(3,1,3),(4,1,2)$, or $T=\POmega_n(q)$ and
$(n,m,q,\varepsilon)=(7,3,3,+)$. In each of these cases, we see
from~\cite[Proposition~$4.1.4$,~$4.1.6$]{KL} 
that $T_\omega$ is a $\{2,3\}$-group. So $\mu(T_\omega)\leq 2$ and the
result follows from Lemma~\ref{id2}. Moreover, if~$(i)$ or~$(ii)$
holds, then from~\cite[Proposition~$4.1.3$--$4$,~$4.1.6$]{KL}
$\cent{T_\omega}{\E{T_\omega}}$ is 
soluble and hence  the theorem follows from Proposition~\ref{emb2}
or~\ref{emb1} respectively.    

Now suppose that $T=\PSL_n(q)$ and that $M$ is of type
$P_{m,n-m}$. From~\cite[Proposition~$4.1.22$]{KL}, we see that $M$ contains
a parabolic subgroup 
(not necessarily maximal) of $T$. Therefore $T_\omega$
contains a Borel subgroup of $T$ and so, $T_\omega$ contains the
normalizer of a Sylow $p$-subgroup of $T$. Now from
Proposition~\ref{sylow}, every non-trivial subdegree of $T$ is
divisible by $p$. 

It remains to consider the case that $T=\POmega_{n}^{\pm}(q)$ and $M$
is of type $\Sp_{n-2}(q)$ with $q$
even. From~\cite[Proposition~$4.1.7$]{KL}, we see that 
$\cent {T_\omega}{\E{T_\omega}}$ is soluble and that
  $\E{T_\omega}/\Z{\E{T_\omega}}\cong \PSp_{n-2}(q)$ is
simple. Therefore each set of pairwise coprime non-trivial subdegrees of $T$ has size at most two, by Proposition~\ref{emb1}.

\smallskip

\noindent\textsc{Case $M\in \mathcal{C}_2$: }$M$ is the stabilizer of a
direct sum decomposition.

\noindent We first consider the case that $M$ is of type $\GL_{n/2}(q^2).2$ if
$T=\PSU_n(q)$, of type $\GL_{n/2}(q).2$ if $T=\PSp_n(q)$, of type
$\GL_{n/2}(q).2$ 
or $\OO_{n/2}(q)^2$ (with $n/2\geq 5$ odd) if $T=\POmega_n^+(q)$, and
of type $\OO_{n/2}(q)^2$ (with $n/2\geq 5$ odd) if 
$T=\POmega_n^-(q)$. From~\cite[Proposition~$4.2.4$--$5$,~$4.2.7$,~$4.2.16$]{KL},
we see 
 that $\cent {T_\omega}{\E{T_\omega}}$ is soluble, and that either
$\E{T_\omega}/\Z{\E{T_\omega}}\cong S^\ell$ for some nonabelian simple
group $S$ (here $\ell=1$ or $2$) or $T_\omega$ is soluble. In the
former case, from Proposition~\ref{emb1}, each set of pairwise coprime non-trivial subdegrees of $T$ has size at most two.
The latter case occurs 
only for $T=\PSp_4(3)$, which we excluded from this analysis.

In the rest of  the proof of this case we use the detailed information
on the Sylow normalizers of the Lie type groups in~\cite[Section~$5$]{GM}.
Given a connected reductive algebraic group $\bf{G}$ defined over a finite
field $\mathbb{F}_q$ and $F:\bf{G}\to \bf{G}$ the corresponding Frobenius
endomorphism, we adopt the terminology in~\cite{GM} for the Sylow $\Phi_e$-tori
of $\bf{G}$ and we refer to as Sylow $\Phi_e(q)$-tori their subgroups of
fixed points (under $F$) in the finite Lie type group $G=\bf{G}^F$.
Furthermore, we deal
with each family of classical groups separately. In fact, although the
arguments are very similar in every case, there are some slight
differences that can be presented neatly only by dealing with one family
at a time. 

\smallskip\noindent
\textsc{The groups }$T=\PSL_n(q)$. Assume that $M$ is of type
$\GL_m(q)\wr 
\Sym(n/m)$ with $m\geq 1$. Let 
$\mathbb{F}_q^n=V_1\oplus\cdots\oplus V_{n/m}$ be the direct sum
decomposition preserved by $T_\omega$ and let $H$ be
the normal subgroup of $T_\omega$ fixing every direct summand $V_i$, for
$i\in \{1,\ldots,n/m\}$. If $m\geq 3$, or $m=2$ and $q\geq 4$, we see 
from~\cite[Proposition~$4.2.9$]{KL} that
$\cent{T}{T_\omega}$ is soluble and that
$\E{T_\omega}/\Z{\E{T_\omega}}$ is isomorphic to a direct product of
pairwise isomorphic nonabelian simple groups. So from
Proposition~\ref{emb1}, each set of pairwise coprime non-trivial subdegrees of $T$ has size at most two. This leaves the cases $m=1$, and $(m,q)=(2,2)$ and $(2,3)$.

Assume next that $m=1$. From the structure and from the order of
$T_\omega$ we see that $T_\omega$ is the normalizer of a Sylow
$\Phi_1(q)$-torus $S_1$ of $T$, that is, $T_\omega=\N T {S_1}$. Recall that $n\geq 3$. Let $r$ be the
largest prime dividing $q-1$. Now, if $r>3$, or if $r=2$ and $q\equiv
1\mod 4$, then 
from~\cite[Theorems~$5.14$   and~$5.19$]{GM}  we obtain that $\N T
{S_1}$ contains the normalizer of a Sylow $r$-subgroup of $T$. I this case
every non-trivial subdegree of $T$ is 
divisible by $r$ by Proposition~\ref{sylow}. It remains to consider
the case that either $q=2$, or $2$ and $3$ are 
the only primes dividing $q-1$. Assume that $q=2$. If $n\leq 4$, then
$T_\omega$ is a $\{2,3\}$-group, so $\mu(T_\omega)\leq 2$ and the result
follows from Lemma~\ref{id2}. Suppose that $n\geq 5$ and let $N$ be
the last term of the derived series of $T_\omega$. From
Lemma~\ref{derivedSeries}, the group $N$ fixes a unique point of
$\Omega$, and the result follows from Lemmas~\ref{id1}
and~\ref{id2}. So, we may now assume that $q\neq 2$. If $3$ divides
$q-1$, then    
from~\cite[Theorems~$5.14$]{GM}  we obtain that either $\N T
{S_1}$ contains the normalizer of a Sylow $3$-subgroup of $T$ or
$n=3$. In the former case, every non-trivial subdegree of $T$ is 
divisible by $3$ from Proposition~\ref{sylow}. In the latter case, as
$q-1$ is only divisible by 
the primes $2$ and $3$, we have that $T_\omega$ is a $\{2,3\}$-group and the
result follows from Lemma~\ref{id2}. Therefore, it remains to deal
with the case that $2$ is the only prime dividing $q-1$  and $q\equiv 3\mod
4$, that is, $q=3$. We do this in the following paragraph.

Assume $(m,q)=(1,3)$, $(2,2)$ or $(2,3)$. If $n/m\leq 4$, then
$T_\omega$ is a $\{2,3\}$-group and the result follows from
Lemma~\ref{id2}. Suppose that $n/m\geq 5$. From
Proposition~\ref{claaa1}, the kernel $H$ of the $T_\omega$-action on the direct summands $V_i$ of $V$ fixes a unique point of 
$\Omega$. In each case $H$ is a $\{2,3\}$-group and hence $\mu(H)\leq 2$ and
the result follows from Lemma~\ref{id2}.
The analysis for the remaining classical groups is very similar.

\smallskip\noindent
\textsc{The groups } $T=\PSU_n(q)$. Assume that $M$ is of type $\GU_m(q)\wr
\Sym(n/m)$ with $m\geq 1$. Let 
$\mathbb{F}_{q^2}^n=V_1\oplus\cdots\oplus V_{n/m}$ be the direct sum
decomposition preserved by $T_\omega$ and let $H$ be
the normal subgroup of $T_\omega$ fixing every direct summand $V_i$, for
$i\in \{1,\ldots,n/m\}$.  If $m\geq 4$, or if $m=3$ and $q\geq 3$, or
if $m=2$ and $q\geq 4$, we see
from~\cite[Proposition~$4.2.9$]{KL}  that $\cent
T{\E{T_\omega}}$ is soluble and that $\E{T_\omega}/\Z{\E{T_\omega}}$
is isomorphic to the direct product of pairwise isomorphic nonabelian
simple groups. So from Proposition~\ref{emb1}, $T$ has at most $2$
non-trivial coprime subdegrees. We now consider the remaining cases, namely, $m=1$ and $(m,q)=(2,2), (2,3)$ and $(3,2)$. 

Assume that $m=1$. Now the order of $\GU_1(q)$ is divisible by $q+1$
and so  $T_\omega$ is the normalizer of a Sylow $\Phi_2(q)$-torus
$S_2$ of $T$. Set $r=2$ if $q\equiv 3\mod 4$, or choose the largest
prime $r>2$ dividing $q+1$ if $q\not\equiv 3\mod
4$. From~\cite[Theorem~$5.14$,~$5.19$]{GM}, we have that either
$T_\omega$ contains the normalizer of a Sylow $r$-subgroup of $T$ (and
hence every non-trivial subdegree of $T$ is divisible by $r$ from
Proposition~\ref{sylow}) or $n=3$ and $r=3$. In the latter case, by
our choice of $r$, the only primes dividing $q+1$ are $2$ and
$3$. Since $n=3$, we obtain that $T_\omega$ is a $\{2,3\}$-group and
by Lemma~\ref{id2}, $T$ has at most $2$ non-trivial coprime
subdegrees. 

Assume that $(m,q)=(2,2)$, $(2,3)$ or $(3,2)$.  If $n/m\leq 4$, then
$T_\omega$ is a $\{2,3\}$-group and the result follows from
Lemma~\ref{id2}. Suppose that $n/m\geq 5$. From
Proposition~\ref{claaa1}, the group $H$ fixes a unique point of
$\Omega$. As $H$ is a $\{2,3\}$-group, we obtain $\mu(H)\leq 2$ and
the result follows from Lemma~\ref{id2}.

\smallskip\noindent
\textsc{The groups } $T=\PSp_n(q)$. Assume that $M$ is of
type $\Sp_m(q)\wr 
\Sym(n/m)$ with $m\geq 2$ even.   Let 
$\mathbb{F}_{q}^n=V_1\oplus\cdots\oplus V_{n/m}$ be the direct sum
decomposition preserved by $T_\omega$ and let $H$ be
the normal subgroup of $T_\omega$ fixing every direct summand $V_i$, for
$i\in \{1,\ldots,n/m\}$. If $m\geq 4$, or if $m=2$ and $q\geq 4$, we see
from~\cite[Proposition~$4.2.10$]{KL}  that $\cent
T{\E{T_\omega}}$ is soluble and that $\E{T_\omega}/\Z{\E{T_\omega}}$
is isomorphic to a direct product of pairwise isomorphic nonabelian
simple groups. So from Proposition~\ref{emb1}, $T$ has at most $2$
non-trivial coprime subdegrees. We now consider the remaining cases. 

Assume that  $(m,q)=(2,2)$ or $(2,3)$. If $n/m\leq 4$, then $T_\omega$
is a $\{2,3\}$-group and hence the result follows from
Lemma~\ref{id2}.  Suppose that $n/m\geq 5$. From
Proposition~\ref{claaa1}, the group $H$ fixes a unique point of
$\Omega$. As $H$ is a $\{2,3\}$-group, we obtain $\mu(H)\leq 2$ and
the result follows from Lemma~\ref{id2}.

\smallskip\noindent
\textsc{The groups }$T=\POmega_n(q)$ ($n$ odd). Assume that $M$ is of
type $\OO_m(q)\wr 
\Sym(n/m)$ with $m\geq 1$ (where $q=p\geq 3$ if $m=1$). Let 
$\mathbb{F}_{q}^n=V_1\oplus\cdots\oplus V_{n/m}$ be the direct sum
decomposition preserved by $T_\omega$ and let $H$ be
the normal subgroup of $T_\omega$ fixing every direct summand $V_i$, for
$i\in \{1,\ldots,n/m\}$. If $m\geq 5$,
or if $m=3$ and $q\neq 3$, we see
from~\cite[Proposition~$4.2.12$]{KL}  that $\cent
T{\E{T_\omega}}$ is soluble and that $\E{T_\omega}/\Z{\E{T_\omega}}$
is isomorphic to a direct product of pairwise isomorphic nonabelian
simple groups. So from Proposition~\ref{emb1}, $T$ has at most $2$
non-trivial coprime subdegrees. We now consider the remaining cases. 

Assume $m=1$ or $(m,q)=(3,3)$. Note that $\OO_1(q)$ has order $2$ and
is generated by $-1$. If $n/m\leq 4$, then $T_\omega$
is a $\{2,3\}$-group and hence the result follows from
Lemma~\ref{id2}.  Suppose that $n/m\geq 5$. From
Proposition~\ref{claaa1}, the group $H$ fixes a unique point of
$\Omega$. As $H$ is a $\{2,3\}$-group, we obtain $\mu(H)\leq 2$ and
the result follows from Lemma~\ref{id2}.

\smallskip\noindent
\textsc{The groups } $T=\POmega_n^+(q)$ ($n$ even). Assume that $M$ is of type
$\OO_m^\varepsilon(q)\wr \Sym(n/m)$ with $\varepsilon\in\{\circ,+,-\}$
(where $\varepsilon^{n/m}=+$ if $m$ is even) and with $q=p\geq 3$ if
$m=1$.  Let 
$\mathbb{F}_{q}^n=V_1\oplus\cdots\oplus V_{n/m}$ be the direct sum
decomposition preserved by $T_\omega$ and let $H$ be
the normal subgroup of $T_\omega$ fixing every direct summand $V_i$, for
$i\in \{1,\ldots,n/m\}$. If $m\geq 5$, or if $m=4$ and $q\geq 4$, or
if $m=4$ and $\varepsilon=-$, or if $m=3$ and $q\neq 3$, we see
from~\cite[Proposition~$4.2.11$,~$4.2.14$]{KL}  that $\cent
T{\E{T_\omega}}$ is soluble and that $\E{T_\omega}/\Z{\E{T_\omega}}$
is isomorphic to a direct product of pairwise isomorphic nonabelian
simple groups. So from Proposition~\ref{emb1}, $T$ has at most $2$
non-trivial coprime subdegrees. We now consider the remaining cases. 

Assume that $m=1$, or
$(m,q,\varepsilon)=(3,3,\circ)$, $(4,2,+)$ or $(4,3,+)$ (recall from~\cite[Table~$4.2A$]{KL} that if $m=1$ then $q=p\geq
3$). In each of these cases, $H$ is a 
$\{2,3\}$-group. If 
$n/m\leq 4$, then $T_\omega$ 
is a $\{2,3\}$-group and hence the result follows from
Lemma~\ref{id2}.  Suppose that $n/m\geq 5$. From
Proposition~\ref{claaa1}, the group $H$ fixes a unique point of
$\Omega$, $\mu(H)\leq 2$ and
the result follows from Lemma~\ref{id2}.

Assume $m=2$. Note that if $\varepsilon=-$, then $n/2$ is even because
$\varepsilon^{n/2}=+$. Now,
$\OO_2^{+}(q)$ is a dihedral group of order $2(q-1)$, and $\OO_2^-(q)$ is
a dihedral group of order $2(q+1)$. The 
largest power of the polynomial $x-1$ dividing the generic order of
$\POmega_n^+$ is $n/2$. Similarly, if $n/2$ is even, the
largest power of the polynomial $x+1$ dividing the generic order of
$\POmega_n^+$ is $n/2$.  Therefore, considering the structure of
$T_\omega$ and its order, we obtain that $T_\omega$ is the
normalizer of a $\Phi_1(q)$-torus $S_1$ of $T$ if $\varepsilon=+$ and is the
normalizer of a $\Phi_2(q)$-torus $S_2$ of $T$ if
$\varepsilon=-$.  Assume first that $\varepsilon=-$. 
Set 
$r=2$ if $q\equiv 3\mod 4$, or choose a prime $r$ dividing $q+1$ and
coprime to $q-1$ if 
$q\not\equiv 3\mod 4$. From~\cite[Theorem~$5.14$,~$5.19$]{GM}, $T_\omega$
contains the normalizer of a Sylow $r$-subgroup of $T$. In this case every
non-trivial subdegree of $T$ is divisible by $r$ by
Proposition~\ref{sylow}.  
Assume now that $\varepsilon=+$. Then each $V_i$ is a hyperbolic plane for its stabilizer 
$M_i \cong \OO_2^+(q)$ in $M$. As any hyperbolic plane contains exactly two isotropic lines, then
$M$ is the stabilizer in $A$ of a decomposition of $V$ in $1$-dimensional spaces. So we are back
to the case $m=1$, which has already been considered.   

\smallskip\noindent
\textsc{The groups }$T=\Omega_n^-(q)$ ($n$ even). Assume that $M$ is of type
$\OO_m^\varepsilon(q)\wr \Sym(n/m)$ with $\varepsilon\in\{\circ,-\}$
and with $q=p\geq 3$ if $m=1$. Let 
$\mathbb{F}_{q}^n=V_1\oplus\cdots\oplus V_{n/m}$ be the direct sum
decomposition preserved by $T_\omega$ and let $H$ be
the normal subgroup of $T_\omega$ fixing every direct summand $V_i$, for
$i\in \{1,\ldots,n/m\}$. If $m\geq 4$, or if $m=3$ and $q\ne
3$, we see from~\cite[Proposition~$4.2.11$,~$4.2.14$]{KL}  that $\cent
T{\E{T_\omega}}$ is soluble and that $\E{T_\omega}/\Z{\E{T_\omega}}$
is isomorphic to a direct product of pairwise isomorphic nonabelian
simple groups. So from Proposition~\ref{emb1}, $T$ has at most $2$
non-trivial coprime subdegrees.  

Assume that $m=1$ or $(m,q)=(3,3)$  (recall that if $m=1$ then $q=p\geq
3$). In each of these cases, $H$ is a
$\{2,3\}$-group. If 
$n/m\leq 4$, then $T_\omega$ 
is a $\{2,3\}$-group and hence the result follows from
Lemma~\ref{id2}.  Suppose that $n/m\geq 5$. From
Proposition~\ref{claaa1}, the group $H$ fixes a unique point of
$\Omega$, $\mu(H)\leq 2$ and
the result follows from Lemma~\ref{id2}.

Assume that $m=2$. Note that from~\cite[Table~$3.5F$]{KL}, $n/2$ is
odd. Now, $\OO_2^{-}(q)$ has order divisible by $q+1$. Since $n/2$ is
odd, the largest power of the polynomial $x+1$ dividing the generic
order of $\POmega_n^-$ is $n/2$. Therefore, considering the structure
of $T_\omega$ and its order, we obtain that $T_\omega$ is the
normalizer of a $\Phi_2(q)$-torus of $T$. Set $r=2$ if $q\equiv 3\mod
4$, or choose a prime $r$ dividing $q+1$ and coprime to $q-1$ if
$q\not \equiv 3\mod 4$. From~\cite[Theorem~$5.14$,~$5.19$]{GM}, $T_\omega$
contains the normalizer of a Sylow $r$-subgroup of $T$. So every
non-trivial subdegree of $T$ is divisible by $r$ from
Proposition~\ref{sylow}. 

\smallskip

\noindent\textsc{Case $M\in \mathcal{C}_3:$ }$M$ is the stabilizer of a structure on $V$ as an $n/r$-dimensional space over an extension field of $\mathbb{F}_q$ of prime index $r$.

\noindent From~\cite[Tables~$3.5A$--$F$]{KL}, we see that $M$ is of type
$\GL_m(q^r)$ if $T=\PSL_n(q)$, of type $\GU_m(q^r)$ if $T=\PSU_n(q)$,
of type $\Sp_m(q^r)$ or $\GU_{n/2}(q)$ (with $q$ odd) if
$T=\PSp_n(q)$, of type $\OO_{n/r}(q^r)$ (with $n/r\geq 3$) if
$T=\POmega_n(q)$, of type $\GU_{n/2}(q)$,  $\OO_{n/r}^+(q^r)$ (with $n/r\geq 4$), or $\OO_{n/2}(q^2)$ (with $nq/2$ odd) if $T=\POmega_n^+(q)$, and of type 
$\GU_{n/2}(q)$, $\OO_{n/r}^-(q^r)$ (with $n/r\geq 4$), or $\OO_{n/2}(q^2)$ (with $nq/2$ odd) if
$T=\POmega_n^-(q)$. From~\cite[Section~$4.3$]{KL}, the group
$\cent{T_\omega}{\E{T_\omega}}$ is soluble.  Furthermore, in each of the cases, considering the restrictions on $n$, $q$ and $r$ that we have given above, we  see from Table~\ref{table2} that either $T_\omega$ is soluble or
$\E{T_\omega}/\Z{\E{T_\omega}}\cong S^\ell$ for some nonabelian simple
group $S$ (where $\ell=1$, or $\ell=2$ if $T=\POmega_{4r}^+(q)$ and
$M$ is of type $\OO_4^+(q^r)$). In the latter case, the theorem
follows from Proposition~\ref{emb1}. Assume now that $T_\omega$ is soluble. Since we are excluding
$T=\PSp_4(3)$, with a direct inspection we see that  $T=\PSL_r(q)$ or
$\PSU_r(q)$, and in particular that $r\geq3$. 

From~\cite[Proposition~$4.3.6$]{KL}, the group $T_\omega$
is isomorphic to $\mathbb{Z}_a\rtimes \mathbb{Z}_r$ with
$a=(q^r-\varepsilon)/((q-\varepsilon)\gcd(q-\varepsilon,r))$ (here
$\varepsilon=1$ if $T=\PSL_r(q)$ and $\varepsilon=-1$ if
$T=\PSU_r(q)$). In particular, $T_\omega$ is the normalizer of a $\Phi_1(q)$-torus of $T$ if $T=\PSL_r(q)$ and is the normalizer of a $\Phi_2(q)$-torus of $T$ if $T=\PSU_r(q)$. From Zsigmondy's theorem, we see that there exists a prime $s$ dividing $q^r-\varepsilon$ and coprime to $q^i-\varepsilon$ for every $i\in\{1,\ldots,r-1\}$ (note that $r\geq 3$ is prime and that we are excluding $\PSU_3(2)$ since it is soluble). Clearly, $s\geq 3$. Moreover if $3$ divides $q^r-\epsilon$, then $q-\epsilon\equiv 0\pmod{3}$ if $\epsilon=-$, and $q^2-\epsilon\equiv 0\pmod{3}$ if $\epsilon=+$; since $r\geq3$, this implies that $s\ne3$. Thus $s>3$.  From~\cite[Theorem~$5.14$]{GM}, we obtain that $T_\omega$
contains the normalizer of a Sylow $s$-subgroup of $T$, and hence every non-trivial subdegree of $T$ is divisible by $r$ by
Proposition~\ref{sylow}. 

%In particular, $T_\omega$ is a Singer cycle of order
%$a$ of $T$ extended by its Galois group. As $r\geq 3$ is a prime and
%as $\PSU_3(2)$ is soluble, from Zsigmondy's theorem we have that
%$q^{r}-\varepsilon$ has a primitive prime divisor $s$. Clearly, $s\neq
%r$. Let $S$ be a Sylow $s$-subgroup of $T_\omega$. From the formula
%for the order of $T$, we see that $S$ is a Sylow $s$-subgroup of
%$T$. Also, since $S$ is characteristic in $T_\omega$, we have
%$T_\omega=\N T S$. So, from Proposition~\ref{sylow}, every subdegree
%of $T$ is divisible by $s$. 

\smallskip
 
\noindent\textsc{Case  $M\in \mathcal{C}_4$: }$M$ is the
stabilizer of a tensor product decomposition. 

\noindent From~\cite[Section~$3.5$]{KL}, we get that $M$ is of type
$\GL_m(q)\otimes \GL_{n/m}(q)$ if $T=\PSL_n(q)$ (with $n\neq m^2$), of type
$\GU_m(q)\otimes \GU_{n/m}(q)$ if $T=\PSU_n(q)$ (with $n\neq m^2$), of type
$\Sp_m(q)\otimes \OO_{n/m}^{\varepsilon}(q)$ if $T=\PSp_n(q)$, of type 
$\OO_m(q)\otimes \OO_{n/m}(q)$ (with $n\neq m^2$) if $T=\POmega_n(q)$, of type
$\Sp_m(q)\otimes \Sp_{n/m}(q)$ (with $n\neq m^2$) or
$\OO_m^{\varepsilon_1}(q)\otimes\OO_{n/m}^{\varepsilon_2}(q)$ if
$T=\POmega_n^+(q)$, and of type $\OO_{m}(q)\otimes \OO_{n/m}^-(q)$ if
$T=\POmega_n^-(q)$. Note that if $T=\PSp_n(q)$, then $q$ is odd
(see~\cite[Table~$3.5C$]{KL}). With a direct inspection
 (using~\cite[Proposition~$4.4.10$--$12$,~$4.4.14$,~$4.4.17$,~$4.4.18$]{KL})
we see that $\cent {T_\omega}{\E{T_\omega}}$ is soluble. We claim that
either~$(i):$ $T_\omega$ is soluble, or~$(ii):$  $\E{T_{\omega}}/\Z{\E{T_\omega}}$
 is a direct product of (at least one) isomorphic
simple groups, or ~$(iii):$ each simple direct factor of
$\E{T_{\omega}}/\Z{\E{T_\omega}}$ has multiplicity  one, or $(iv):$ 
$T=\PSp_{4m}(q)$ and $M$ is of type $\PSp_m(q)\otimes
\SL_2(q)\otimes\SL_2(q)$, or $T=\POmega_{4m}^+(q)$ and $M$ is of type 
$\OO_{m}^{\varepsilon}(q)\otimes\SL_2(q)\otimes\SL_2(q)$.

As usual,~$(i)$ 
occurs only for small values of $n$ and
$q$. (Recall that $q$ is odd if
$T=\PSp_n(q)$, and $m,n/m\geq 4$ if $\varepsilon_1=\varepsilon_2=+$ and $T=\POmega_n^+(q)$~\cite[Proposition~$4.14$]{KL}.) 
Namely,~$(i)$ arises when 
$n=6$, $q=2$  
and $T=\PSU_6(2)$, when $n=6,8$, $q=3$ and $T=\PSp_n(q)$ (here $M$ is
of type $\Sp_2(3)\otimes \OO_3(3)$ or $\Sp_2(3)\otimes \OO_4^+(3)$),
and when $n=12$, $q=3$ and 
$T=\POmega_n^+(q)$ (here $M$ is of type
$\OO_3(3)\otimes\OO_4^+(3)$). In 
each of these cases,
from~\cite[Proposition~$4.4.10$,~$4.4.14$,~$4.4.17$]{KL} we see that
$M$ is a $\{2,3\}$-group. Thence $\mu(T_\omega)\leq 2$ and the theorem
follows from Lemma~\ref{id2}. 

Now we consider~$(ii)$. Again this occurs in a small list
of cases, typically when one of the two central factors in the type of
$M$ is soluble. Namely,~$(ii)$ arises for $T=\PSL_n(q)$ when $m=2$, $q=2,3$; for
$T=\PSU_n(q)$ when
$(m,q)=(2,2),(2,3),(3,2)$; for $T=\PSp_n(q)$ when $m=2$ and $q=3$, or
$n=3m$ and $q=3$, or $n=4m$ and $q=3$, or $n=20$ (here $M$ is of type
$\Sp_4(q)\otimes \OO_5(q)$), or $n=6$ (here $M$ is of type
$\Sp_2(q)\otimes \OO_3(q)$), or $n=8$ (here $M$ is of type
$\Sp_2(q)\otimes\OO_4^+(q)$); for $T=\POmega_n(q)$ when $m=3$, $q=3$ (recall that $n\neq m^2$);
for $T=\POmega_n^+(q)$ when $m=2$ and $q=2,3$, or $m=3$ and 
$q=3$ (here $M$ is of type $\OO_3(q)\otimes
\OO_{n/3}^{\varepsilon_2}(q)$), or $m=4$ and $q=3$ (here $M$ is of type
$\OO_4(q)^+\otimes \OO_{n/4}^{\varepsilon}(q)$); for $T=\POmega_n^-(q)$
when $m=3$, $q=3$. In each of
these cases, from Proposition~\ref{emb1} the result follows.

Now, if $\E{T_\omega}/\Z{\E{T_\omega}}$ is as in~$(iii)$, then from
Proposition~\ref{emb2} the group $T$ has at most two non-trivial
coprime subdegrees.

Finally, with a direct inspection on the type of $M$, we see that if
$(i)$,~$(ii)$ and~$(iii)$ do not hold, then both central 
factors of $M$ are insoluble and one of the two is the central product
of smaller quasisimple groups. From Table~\ref{table2}, this happens
only when $\OO_4^+(q)$ is one of the central factors of $M$. Now
from~\cite[Table~$3.5A-F$]{KL}, we see that either $T=\PSp_{4m}(q)$
or $T=\POmega_{4m}^+(q)$ and our claim is proved. In~$(iv)$ we
may use Proposition~\ref{emb2} to conclude that $T$ has at most two
non-trivial coprime subdegrees.

\smallskip

\noindent\textsc{Case $M\in \mathcal{C}_5$: } $M$ is the stabilizer of a
subfield of $\mathbb{F}_q$ of prime index $r$.  

\noindent From~\cite[Section~$3.5$]{KL}, $M$ is of type
$\GL_n(q^{1/r})$ if $T=\PSL_n(q)$, of type $\GU_n(q^{1/r})$,
$\OO_n^{\varepsilon}(q)$ or $\Sp_n(q)$ if $T=\PSU_n(q)$, of type
$\Sp_n(q^{1/r})$ if $T=\PSp_n(q)$, of type $\OO_n(q^{1/r})$ if
$T=\POmega_n(q)$, of type $\OO_n^+(q^{1/r})$ or $\OO_n^-(q^{-1})$ if
$T=\POmega_n^+(q)$, and of type $\OO_n^-(q^{1/r})$ if
$T=\POmega_n^{-}(q)$. Since we are excluding the cases $T=\PSU_3(3)$
and $\PSU_4(3)$, the group $\E{T_\omega}/\Z{\E{T_\omega}}$ is either
simple, or a direct product of two isomorphic simple groups (which
occurs when $T=\PSU_4(q)$ and $M$  is of type $\OO_4^+(q)$), or
$T=\PSU_3(2^r)$. In the third case, we see
from~\cite[Proposition~$4.5.3$~(II)]{KL}, that $M$ is a
$\{2,3\}$-group, and then $\mu(T_\omega)\leq 2$ and the result follows
from Lemma~\ref{id2}. In the remaining cases,
from~\cite[Proposition~$4.5.3$--$6$,~$4.5.8$,~$4.5.10$]{KL},
we see that $\cent {T_\omega}{\E{T_\omega}}$ is soluble and so, from
Proposition~\ref{emb1}, $T$ has at most two non-trivial coprime
subdegrees. 

\smallskip

\noindent\textsc{Case $M\in\mathcal{C}_6$: }$M$ is the
normalizer of an extraspecial $r$-group in an absolutely irreducible
representation. 

\noindent From~\cite[Section~$3.5$]{KL}, the group $M$ is of type $r^{2m}\Sp_{2m}(r)$ if
$T=\PSL_n(q)$ or $T=\PSU_n(q)$ (with $n=r^m$), of type $2^{1+2m}\OO_{2m}^-(2)$ if
$T=\PSp_n(q)$ (with $n=2^m$), and of type $2_+^{1+2m}\OO_{2m}^+(2)$ if
$T=\POmega_n^+(q)$ (with $n=2^m$). From~\cite[Proposition~$4.6.5$--$6$,~$4.6.8$--$9$]{KL},
we see that $\cent{T_\omega}{\E{T_\omega}}$ is soluble. Furthermore, since we are excluding the group $\PSL_2(q)$ (which we studied in the first part of the proof), from Table~\ref{table2} we have that $T_\omega$ is soluble if and only if $T=\PSL_3(q),\PSU_3(q)$. (Recall that $n\geq 4$ if $T=\PSp_n(q)$ and $n\geq 8$ if $T=\POmega_n^+(q)$.) If $T=\PSL_3(q)$ or $\PSU_3(q)$, then with a direct
inspection of the structure of $M$ described
in~\cite[Proposition~$4.6.5$--$6$]{KL}, we see
that $M$ is a $\{2,3\}$-group and so $\mu(T_\omega)\leq
2$. Hence from Lemma~\ref{id2} the group $T$ has at most two
non-trivial coprime subdegrees. 

It remains to consider the case that $M$ is insoluble. Let $N$ be the
last term of the derived series of $M$. Since $M/T_\omega$ is soluble, we have $N\leq T_\omega$. Furthermore, from
the group structure of $M$, the group $N$ contains a
characteristic $r$-subgroup $R$ with $N/R\cong \Sp_{2m}(r)$ if
$T=\PSL_n(q),\PSU_n(q)$, with $N/R\cong (\OO_{2m}^-(2))'$ if
$T=\PSp_n(q)$, and with $N/R\cong (\OO_{2m}^+(2))'$ if
$T=\POmega_n^+(q)$. 

From Lemma~\ref{derivedSeries}, the group $N$ fixes only the point
$\omega$ of $\Omega$. We show that
$\mu(N)\leq 2$, from which the theorem follows (in this case) from
Lemma~\ref{id2}. From Lemma~\ref{lucky}, we have $\mu(N)\leq 3$. If $\mu(N)\leq 2$,
then the result follows from Lemma~\ref{id2}. Suppose that $\mu(N)=3$ and
 let $A_1,A_2,A_3$ be three maximal 
subgroups of $N$ having pairwise relatively prime indices in
$N$. Let $U$ be
the normal subgroup of $N$ with $R\leq U$ and with $N/U$ simple (that
is, $U/R=\Z{N/R}$). From Lemma~\ref{lucky}, relabelling the $A_i$ if necessary, 
we have that $r$ divides
$|N:A_3|$ and that $N=(A_1/U)(A_2/U)$ is a
maximal factorization of the 
simple group $N/U$ with $\gcd(|N:A_1|,|N:A_2|)=1$. Therefore
$(N/U,A_1/U,A_2/U)$ is one of the triples in~\cite[Table~$1$]{DGPS}.  
Suppose that $T=\PSL_n(q)$ or $\PSU_n(q)$, that is, $N/U\cong
\PSp_{2m}(r)$. From~\cite[Table~$1$]{DGPS}, we see that $r$ divides $|N:A_1|$ or $|N:A_2|$, contradicting the fact that $|N:A_3|$ is coprime with $|N:A_1|$ and with $|N:A_2|$. Now suppose that
$T=\PSp_n(q)$, that is, $N/U\cong
\POmega_{2m}^-(2)$. (Recall that $r=2$.) From~\cite[Table~$1$]{DGPS} we see that $\POmega_4^-(2)=\PSL_2(4)$ and $\POmega_{6}^-(2)=\PSU_4(2)$ are the only orthogonal groups $\POmega_n^-(2)$ admitting a coprime
factorization. Furthermore, $2$ divides $|N:A_1|$ or $|N:A_2|$, a contradiction. Finally suppose that
$T=\POmega_n^+(q)$, that is, $N/U\cong
\POmega_{2m}^+(2)$. From~\cite[Table~$1$]{DGPS}, we see that $2$ divides $|N:A_1|$ or $|N:A_2|$, again a contradiction because $r=2$.

\smallskip

\noindent\textsc{Case $M\in \mathcal{C}_7$: }$M$ is the stabilizer of a
homogeneous tensor decomposition of $V$. 

\noindent From~\cite[Section~$3.5$]{KL}, we have that the group $M$ is of type
$\GL_m(q)\wr\Sym(t)$ if $T=\PSL_n(q)$ (with $m\geq 3$), of type
$\GU_m(q)\wr \Sym(t)$ if $T=\PSU_n(q)$ (with $m\geq 3$ and $(m,q)\neq
(3,2)$), of type $\Sp_m(q)\wr\Sym(t)$ if $T=\PSp_n(q)$ (with $qt$ odd,
$m\geq 2$ and 
$(m,q)\neq (2,3)$), of type $\OO_m(q)\wr\Sym(n/m)$ if $T=\POmega_n(q)$
(with $m\geq 
3$ and $(m,q)\neq (3,3)$), of type $\Sp_m(q)\wr\Sym(t)$ (with $m\geq
2$ and $(m,q)\neq (2,2),(2,3)$) or $\OO_m^\varepsilon(q)\wr\Sym(t)$
(with $q$ odd, and $m\geq 6$ if $\varepsilon=+$ and $m\geq 4$ if
$\varepsilon=-$) if $T=\POmega_n^+(q)$. In particular, from
Table~\ref{table2} we see that 
$\E{T_\omega}/\Z{\E{T_\omega}}$ is a direct product of isomorphic
simple groups. Furthermore,
from~\cite[Proposition~$4.7.3$--$5$,~$4.7.6$--$8$]{KL} we see that
$\cent{T_\omega}{\E{T_\omega}}$ is soluble. Now as usual from
Proposition~\ref{emb1}, we obtain that $T$ has at most two non-trivial
coprime subdegrees.

\smallskip

\noindent\textsc{Case $M\in \mathcal{C}_8$: }$M$ is a classical
subgroup. 

\noindent From~\cite[Section~$3.5$]{KL}, we see that the group $M$ is of type $\Sp_n(q)$, $\OO_n^{\varepsilon}(q)$ or $\SU_n(q^{1/2})$ if $T=\PSL_n(q)$, and of type $\OO_n^{\pm }(q)$ if $T=\PSp_n(q)$ and $q$ is even. 

Since we are excluding the cases $T=\PSL_3(3)$, $\PSL_3(4)$,
$\PSL_4(2)$, $\PSL_4(3)$ and $\PSp_4(2)$, the group
$\E{T_\omega}/\Z{\E {T_\omega}}$ is either simple or a direct
product of two isomorphic simple groups (in fact, the latter case
occurs when $T=\PSL_4(q)$ or $\PSp_4(q)$ and $M$  is of type
$\OO_4^+(q)$). From~\cite[Proposition~$4.8.3$--$6$]{KL}, we see that
$\cent T{\E {T_\omega}}$ is soluble and so, from
Proposition~\ref{emb1}, $T$ has at most two non-trivial coprime
subdegrees on $\Omega$.  

\smallskip

\noindent\textsc{Case $M\in \mathcal{S}$.}

\noindent Since $M/T_\omega$ is soluble, we have $\E M=\E{T_\omega}$. From the definition of the class $\mathcal{S}$ in~\cite[Chapter~1]{KL}, we have that $\E{M}$ is a nonabelian simple group and $\cent {\Aut(T)} {\E M}=1$. Thus $\E{T_\omega}$ is simple and $\cent T{\E{T_\omega}}=1$. In particular, from Proposition~\ref{emb1}, $T$ has at most two
non-trivial coprime subdegrees on $\Omega$. The proof of Conjecture A' for finite classical groups is now complete.
\end{proof}

\section{Exceptional groups of Lie type}\label{exceptionalLie}

\begin{proof}[Proof of Theorem~A for the exceptional groups of Lie type]
Write $q=p^f$ for some prime $p$ and some $f\geq 1$. The group $T$ is one of the following exceptional simple groups:
$\FFF_4( q) $, $\G_ 2 (q) $ (with $q>2$), $\EEE_6 (q) $, $\EEE_7 (q)$, $\EEE_8 (q) $, ${^2}\BBB_2 (q) $
(with $p=2$ and $f=2f'+1$, where $f'\geq 1$), ${^3\DDD}_4(q)$,
${^2}\G_2(q) $ (with $p=3$ and $f=2f'+1$, where $f'\geq 1$), ${^2\FFF}_4(q)$ (with $p=2$ and $f\geq 2$) and ${^2}\EEE_6 (q) $. The group ${^2}F_4(2)$ is not simple and the Tits group ${^2}F_4(2)'$ will be considered in Section~\ref{sporadicsec} together with the sporadic simple groups.

For the proof of this result
we use~\cite{LS}. Liebeck and Seitz~\cite[Theorem~$2$]{LS} give a
reduction theorem to describe the maximal subgroups of the finite exceptional
groups (and their automorphism groups) similar to the well-known
result of Aschbacher \cite{Aschbacher} for the finite classical groups. They show that $M$ is either in one of five well specified
families listed in~\cite[Theorem~2 $(a)$--$(e)$]{LS} or is contained in the 
automorphism group of a finite simple group. In the latter case, as
$M/T_\omega$ is soluble, the group $\FF{T_\omega}$ is simple and 
the theorem follows from Remark~\ref{ps-max} and Proposition~\ref{emb1}. This shows that in
the rest of this proof we may assume that $M$ is in one of the five
families described in~\cite[Theorem~2 $(a)$--$(e)$]{LS}.

\smallskip

\noindent\textsc{The group $M$ is as  in~\cite[Theorem~$2$~$(a)$]{LS}. }

\noindent In this case, $M=\N A D$, where $D$ is either a parabolic
subgroup of $T$ or $D$ is given in~\cite[Theorem,
  Table~$5.1$ and~$5.2$]{LSS}. In the former case, $T_\omega$
contains a parabolic subgroup of $T$ and hence a Borel subgroup of
$T$. In particular, $T_\omega$ contains the normalizer of a Sylow
$p$-subgroup of $T$ and the theorem follows from
Proposition~\ref{sylow}.

Assume that $D$ is as in~\cite[Table~$5.1$]{LSS}. Now the structure of
$T_\omega$ is described in the second column
of~\cite[Table~$5.1$]{LSS}. With a direct inspection we see that in
each case 
$\cent {T_\omega}{\E{T_\omega}}$ is soluble and either $(i)$:
$\E{T_\omega}/\Z{\E{T_\omega}}$ is the direct product of pairwise
isomorphic simple groups, or $(ii)$: $\E{T_\omega}/\Z{\E{T_\omega}}$ is
the 
direct product of simple groups having multiplicity at most $3$ and
with a unique factor of largest order, or $(iii)$: $T=\G_2(3)$ and
$T_\omega$ is of type $2.(L_2(3)\times L_2(3)).2$, or $(iv)$:
$T=\EEE_7(3)$ and $T_\omega$ is of type $2^3.(L_2(3))^7.2^4.L_3(2)$,
or $(v)$: $T=\EEE_8(3)$ and $T_\omega$ is of type
$2^4.(L_2(3))^8.2^4.\mathrm{AGL}_3(2)$ (here we are using the notation
in~\cite[Table~$5.1$]{LSS}). In particular, in~$(i)$ and~$(ii)$ the
theorem  
follows from Proposition~\ref{emb1} and~\ref{emb2}
respectively. In~$(iii)$, we see 
that $T_\omega$  is a $\{2,3\}$-group, $\mu(T_\omega)\leq 2$ and the
result follows from Lemma~\ref{id2}. Now assume that $T$ and 
$T_\omega$ are as in~$(iv)$ or~$(v)$. 
Then $T_\omega$ contains a Sylow $2$-subgroup of $T$. 
As the Sylow $2$-subgroups of $T=\EEE_7(3)$ and $T=\EEE_8(3)$ are
self-normalizing (see~\cite[Theorem~$6$]{KM} or~\cite[Corollary]{Ko}), 
then we are done by Propositon~\ref{sylow}.

Assume that $D$ is as in~\cite[Table~$5.2$]{LSS}. Suppose that $T$ is
not a Suzuki group or a Ree group, that is, $T$ is not
${^2}\BBB_2(q)$, ${^2}\FFF_4(q)$ or ${^2}\G_2(q)$. Then with a direct
inspection on the order of $T$ and on~\cite[Table~$5.2$]{LSS}, we see
that $T_\omega$ is the normalizer of a Sylow $\Phi_e(q)$-torus
of $T$, for some $e$. For instance, in the last row of~\cite[Table~$5.2$]{LSS}, we
have that $T=\EEE_8(q)$ and $T_\omega=T\cap \N A D$ where $D$ is a torus
of $T$ of order $(q^2-q+1)^4$. In particular, since $\Phi_6(q)=q^2-q+1$
and since $4$ is the largest power of the polynomial $x^2-x+1$ dividing the
generic order of  $\EEE_8$, we obtain that $D$ is a
$\Phi_6(q)$-torus of $\EEE_8(q)$. Suppose that $q^e-1$ has a primitive
prime divisor $r$ with $r\geq 3$. It follows
from~\cite[Theorem~$5.14$]{GM} that either $T_\omega$ 
contains the normalizer of a Sylow $r$-subgroup of $T$, or
$T=\G_2(q)$, $r=3$ and $q\equiv 2,4,5$, or $7\mod 9$. In the former case, every
non-trivial subdegree of $T$ is divisible by $r$, by Proposition~\ref{sylow}. For the latter case,
we note that in~\cite[Table~$5.2$]{LSS} we have $q=3^f$ if
$T=\G_2(q)$. Hence $3$ does not divide $q^e-1$ and the latter case does not arise. It remains to
consider the case that either $q^e-1$ has no primitive prime divisors,
or $2$ is the only primitive prime divisor of $q^e-1$. Clearly, this
happens if and only if $e=2$ and $q+1$ is a power of $2$, or
$(e,q)=(6,2)$, or $e=1$ and $q-1$ is a power of $2$. Suppose that
$e=1$ and $q\equiv 1\mod 4$, or $e=2$ and $q\equiv 3\mod 4$. It follows
from~\cite[Theorem~$5.19$]{GM} that $T_\omega$ 
contains the normalizer of a Sylow $2$-subgroup of $T$ and hence, from
Proposition~\ref{sylow}, every
non-trivial subdegree of $T$ is divisible by $2$. Therefore, it
remains to consider the case that $(e,q)=(6,2)$ or $(1,3)$. Suppose
$(e,q)=(1,3)$. A direct inspection in~\cite[Table~$5.2$]{LSS} shows
that if $D$ is a $\Phi_1(q)$-torus of $T$, then $q>3$ (see the
``Condition'' column in~\cite[Table~$5.2$]{LSS}). Suppose that
$(e,q)=(6,2)$. Again a direct inspection in~\cite[Table~$5.2$]{LSS} shows
that if $D$ is a $\Phi_6(q)$-torus of $T$ (that is, $D$ has order a
power of $q^2-q+1$), then $q=2$ is permitted only if $T={^3}\DDD_4(q)$ (see the
``Condition'' column in~\cite[Table~$5.2$]{LSS}). Now, if
$T={^3}\DDD_4(q)$, $q=2$ and $T_\omega$ is the normalizer of a
$\Phi_6(q)$-torus of $T$, then from~\cite[Table~$5.2$]{LSS} we see
that $T_{\omega}$ is a $\{2,3\}$-group and the theorem follows from
Lemma~\ref{id2}.

Suppose that $T$ is a Suzuki group or a Ree group. Malle
in~\cite[Section~$8$]{GM} investigates the Sylow normalizers of
$T$. We use the notation and the terminology
from~\cite[Section~$8$]{GM}. Then with a direct
inspection of the order of $T$ and of~\cite[Table~$5.2$]{LSS}, we see
that $T_\omega$ is the normalizer of a Sylow $\Phi^{(r)}(q)$-torus
of $T$, for a suitable prime $r$ different from the defining
characteristic of $T$. 
It follows
from~\cite[Theorem~$8.4$]{GM} that either $(i):$ $T_\omega$ 
contains the normalizer of a Sylow $r$-subgroup of $T$, or $(ii):$ 
$T={^2}\G_2(3^{2f+1})$, $r=2$ and $D$ is the torus of size $q+1$, or $(iii):$
$T={^2}\FFF_4(2^{2f+1})$, $r=3$, $D$ is the torus of size $(q+1)^2$ and $2^{2f+1}\equiv 2,5\mod 9$. In~$(i)$, every
non-trivial subdegree of $T$ is divisible by $r$, by Proposition~\ref{sylow}. Suppose that $(ii)$
holds. We may assume that $2$ is the only prime dividing $q+1$
(otherwise we may apply~\cite[Theorem~$8.4$]{GM} to a prime $r'\neq 2$
dividing $q+1$ and we obtain that $T_\omega$ contains the normalizer of a Sylow
$r'$-subgroup). Now, as $q=3^{2f+1}$, we have that $q+1$ is a power of $2$ only if $f=0$, that is, $T={^2}\G_2(3)$ (which we excluded from our analysis). Finally assume
that $(iii)$ holds. Here we have $\Phi^{(r)}(q)=q+1$. Also, again arguing as in~$(ii)$ we may assume
that $q+1$ is a power of $3$. Now,~\cite[Table~$5.1$]{LSS} shows that
$T_{\omega}$ is a $\{2,3\}$-group and the result follows from Lemma~\ref{id2}.

\smallskip

\noindent\textsc{The group $M$ is as in~\cite[Theorem~$2$~$(b)$]{LS}. }

\noindent We have $M=\N A E$, where $E$ is the elementary abelian
$r$-group given in~\cite[Theorem~$1$~(II)]{CLSS} (here $r\neq
p$). We have $T_\omega= \N T E$. The pair $(T,E)$ and the structure of
$\cent T E$ and of 
$\N T E$ are as in~\cite[Table~$1$]{CLSS}. We have nine rows to
consider. If $(T,E)$ is in the~$5^\textrm{th}$,~$8^\textrm{th}$
or~$9^\textrm{th}$  row 
of~\cite[Table~$1$]{CLSS}, then $\E{T_\omega}$ is simple, $\cent
{T_\omega}{\E{T_\omega}}$ is soluble and the result follows from
Proposition~\ref{emb1}. Assume that $(T,E)$ is in the~$2^\textrm{nd}$ row
of~\cite[Table~1]{CLSS}, that is, $T={^2}\G_2(3)'$. As
$T\cong\PSL_2(8)$, the proof in this case was given in Section~\ref{classical}.

Finally, suppose that $(T,E)$ is one of the remaining cases:
$1^\textrm{st}$,~$3^\textrm{rd}$,~$4^\textrm{th}$,~$6^\textrm{th}$
or~$7^\textrm{th}$ row of~\cite[Table~$1$]{CLSS}. With a direct
inspection we see that $T_\omega$ contains a normal $r$-subgroup $R$
with $T_\omega/R$ a simple group (note that $\SL_3(2),\SL_3(3),\SL_5(2)$ and $\SL_3(5)$ are simple). We claim that
$\mu(T_\omega)\leq 2$, from which the theorem follows from
Lemma~\ref{id2}. We argue by contradiction and we assume that
$\mu(T_\omega)\geq 3$ and let $\{A_1,A_2,A_3\}$ be three maximal
subgroups of $T_\omega$ having pairwise relatively prime index in
$T_\omega$. From Lemma~\ref{lucky}, relabelling the $A_i$ if necessary, $r$ divides
$|T_\omega:A_3|$ and  $T_\omega/R=(A_1/R)(A_1/R)$ is a maximal coprime
factorization of $T_\omega/R$ (here note that $\Z {T_\omega/R}=1$ because $T_\omega/R$ is simple). Therefore $(T_\omega/R,A_1/R,A_2/R)$ is
one of the triples in~\cite[Table~$1$]{DGPS}. A direct inspection of $T_\omega/R$, of $r$
and of the maximal coprime factorizations of $T_\omega/U$
in~\cite[Table~$1$]{DGPS}, shows that $r$ divides either
$|T_\omega:A_1|$ or $|T_\omega:A_2|$, a contradiction.

\smallskip

\noindent\textsc{The group $M$ is as in~\cite[Theorem~$2$~$(c)$]{LS}. }

\noindent Here $M$ is the centralizer of a graph, field, or
graph-field automorphism of $T$ of prime order $r$
(see~\cite[Definition~$2.5.13$]{GLS} for 
a definition of these terms). In this case, the
structure of $M$ is described in~\cite[Section~$4.4$]{GLS}. Here we
use the notation in~\cite{GLS}. Write $T={^d}\Sigma(q)$, where
$\Sigma$ is the Lie type of $T$, $q$ is the characteristic and
$d=1,2,3$. We first  consider the case that $M$ is the centralizer of
a field automorphism $x$. Recall that ${^2B}_2(2)\cong
5:4$. From~\cite[Proposition~$4.9.1$]{GLS}, we have 
that $\E{M}/\Z{\E{M}}\cong {^d\Sigma}(q^{1/r})$. Since $\Sigma\in
\{\EEE,\FFF,\G,\BBB,\DDD\}$, we obtain that $\E{M}/\Z{\E{M}}$ and
hence $\E{T_\omega}/\Z{\E{T_\omega}}$ is simple except for
$T={^2\BBB}(2^r)$. Furthermore, from~\cite[Chapter~4]{GLS}, the group
$\cent{T_\omega}{\E{T_\omega}}$ is soluble. Therefore, if $T\neq
{^2}\BBB_2(2^r)$, the result follows from Proposition~\ref{emb1}. If
$T={^2}\BBB_2(2^r)$, then $M\cong (5:4)\times r$, $T_\omega\cong 5:4$,
$T_\omega$ is a $\{2,5\}$-group and the result follows from
Lemma~\ref{id2}.

Assume that $x$ is a graph-field
automorphism. Recall that from~\cite[Definition~$2.5.13$]{GLS}, we
have $T=\G_2(q)$, $\FFF_4(q)$ or
$\EEE_6(q)$. From~\cite[Proposition~$4.9.1$]{GLS}, we have $d=1$,
$r=2,3$ and $\E{M}/\Z{\E{M}}\cong {^r}\Sigma(q^{1/r})$. In particular,
$\E{M}/\Z{\E{M}}$ is simple.  Furthermore, from~\cite[Chapter~4]{GLS}, the group
$\cent{T_\omega}{\E{T_\omega}}$ is soluble and so the result follows
from Proposition~\ref{emb1}. 

If remains to study the case that $x$ is a graph
automorphism. 
Recall that
from~\cite[Definition~$2.5.13$~$(b)$,~$(d)$]{GLS} the groups ${^2}\BBB_2(q)$,
${^2}\FFF_4(q)$, ${^2}\G_2(q)$, $\FFF_4(q)$ and $\G_2(q)$ do not
admit graph automorphisms. In particular, $T=\EEE_6(q)$,
${^2}\EEE_6(q)$ or ${^3}\DDD_4(q)$. We consider separately
$T={^3}\DDD_4(2)$ and we 
use~\cite{ATLAS}. With a direct inspection on the maximal subgroups of
$T$, we see that either $T_\omega$ contains the normalizer of a
Sylow subgroup of $T$ (and hence the theorem follows from
Proposition~\ref{sylow}), or $\E{T_\omega}$ is simple and
$\cent{T_\omega}{\E{T_\omega}}$ is soluble (and hence the theorem
follows from Proposition~\ref{emb1}), or $T_\omega$ is a
$\{2,3\}$-group (and hence the theorem follows from Lemma~\ref{id2}). 
Now we continue the proof for the remaining groups. Note that
from~\cite[Sections~$4.5$,~$4.7$ 
  and~$4.9$]{GLS} the group $\cent{T_\omega}{\E{T_\omega}}$ is 
always soluble. 
From~\cite[Proposition~$4.9.2$~$(b)$]{GLS}, we see that for
$T=\EEE_6(q)$ or ${^2}\EEE_6(q)$ we have
$\E{T_\omega}/\Z{\E{T_\omega}}\cong \FFF_4(q)$ if $p=r=2$,  and for
$T=\DDD_4(q)$ or ${^3}\DDD_4(q)$ we have
$\E{T_\omega}/\Z{\E{T_\omega}}\cong \G_2(q)$ if $p=3$. Moreover,
from~\cite[Tables~$4.5.1$ and~$4.7.3A$]{GLS}, we see that for
$T=\EEE_6(q)$ or ${^2}\EEE_6(q)$ we have
$\E{T_\omega}/\Z{\E{T_\omega}}\cong \FFF_4(q)$ or $\mathrm{C}_4(q)$ if
$p\neq 2$ (depending on the conjugacy class of $x$), for $T=\DDD_4(q)$ we have
$\E{T_\omega}/\Z{\E{T_\omega}}\cong \G_2(q)$ if $p\neq 3$, and for
${^3}\DDD_4(q)$ we have
$\E{T_\omega}/\Z{\E{T_\omega}}\cong \PSL_3(q)$ or $\PSU_3(q)$
(depending whether $q\equiv 1\mod 3$ or $q\equiv 
-1\mod 3$ respectively). In particular, in each of these cases (as we
are excluding ${^3}\DDD_4(2)$) we may
use Proposition~\ref{emb1} and the theorem follows.

\smallskip

\noindent\textsc{The group $M$ is as in~\cite[Theorem~$2$~$(d)$]{LS}. }

\noindent In this case, $T=\EEE_8(q)$, $p>5$ and $\FF M=\Alt(5)\times
\Alt(6)$ or $\Alt(5)\times \PSL_2(q)$. Since $M/T_\omega$ is soluble,
we have $\FF M=\FF {T_\omega}$ and hence the theorem follows from
Proposition~\ref{emb2}. 

\smallskip

\noindent\textsc{The group $M$ is as in~\cite[Theorem~$2$~$(e)$]{LS}. }

\noindent In this case, $\FF M=\FF {T_\omega}$ is described in detail
in~\cite[Table~III]{LS}. With a direct inspection, we see that either 
$\FF{T_\omega}$ is the direct product of two nonabelian simple
groups, or $T=\EEE_8(q)$ and $\FF{T_\omega}=\E{T_\omega}\cong \PSL_2(q)\times
\G_2(q)\times \G_2(q)$ with $p>2$ and $q>3$. In the former case, the
theorem follows from Proposition~\ref{emb1} (if the two simple groups
are isomorphic) or Proposition~\ref{emb2} (if the two simple groups are
non-isomorphic).

Suppose that $T=\EEE_8(q)$ and write $N=\FF{T_\omega}\cong \PSL_2(q)\times
\G_2(q)\times \G_2(q)$. Since $T_\omega/N$ is soluble, the group $N$ is clearly the last term of the derived series of $T_\omega$. From Lemma~\ref{derivedSeries} $N$ fixes only the point $\omega$ of $\Omega$. 
We claim that $\mu(N)\leq 2$. Conjecture A' will follow from this claim and Lemma~\ref{id2}.

It remains to prove that $\mu(N)\leq2$. Write $N=S_1\times S_2\times S_3$ with $S_1\cong S_2\cong \G_2(q)$ and $S_3\cong \PSL_2(q)$.
We see from~\cite[Table~$1$]{DGPS} that $\mu(\PSL_2(q))\leq 2$ and that, for any two maximal subgroups $M_1$ and $M_2$ of $\G_2(q)$, the indices $|\G_2(q):M_1|$ and $|\G_2(q):M_2|$ are divisible by a non-trivial common factor. 
Suppose now that $A_1, A_2, A_3$ are maximal subgroups of $N$ of pairwise coprime indices. If $A_3$, say, projects onto each of the three simple direct factors $\{S_1,S_2,S_3\}$ of $N$, then $A_3 = D\times S_3$ where $D$ is a diagonal subgroup of $S_1\times S_2\cong \G_2(q)\times\G_2(q)$. As $|\PSL_2(q)|$ divides $|\G_2(q)|$, this implies that $|N:A_3|$ is divisible by the orders of each of the simple direct factors, a contradiction. If $A_2$ and $A_3$, say, do not project onto $S_1$, then by maximality we have $A_2=B\times S_2\times S_3$ and $A_3=C\times S_2\times S_3$, where $B$ and $C$ are proper subgroups of $S_1$.  Since  $\G_2(q)$ has no two maximal subgroups of pairwise coprime index, we obtain a contradiction. A similar argument applies for $S_2$. Hence, since $\mu(\PSL_2(q))=2$, relabelling the index set $\{1,2,3\}$ if necessary, we have that either $A_1=B_1\times S_2\times S_3$, $A_2=S_1\times B_2\times S_3$ and $A_3=S_1\times S_2\times B_3$ (where, for each $i$, $B_i$ is a
 maximal subgroup of $S_{i}$), or $A_1=B_1\times S_2\times S_3$, $A_2=S_1\times S_2\times B_3$ and $A_3=S_1\times S_2\times B_3'$ (where $B_1$ is a maximal subgroup of $S_{1}$ and $B_3,B_3'$ are maximal subgroup of $S_3$). In the former case, as  $\G_2(q)$ has no two maximal subgroups of pairwise coprime index, $|N:A_1|=|S_1:B_1|$ and $|N:A_2|=|S_2:B_2|$ are divisible by a non-trivial common factor, a contradiction. We now consider the latter case. If the characteristic $p$ divides $|N:A_3|$, then $|N:A_1|$ and $|N:A_2|$ are coprime to $q$ implying that $B_3, B_3'$ are parabolic subgroups of $S_3\cong \PSL_2(q)$, a contradiction. Hence $p$ is coprime with $|N:A_1|$ and $B_1$ is a maximal parabolic subgroup of $S_1\cong \G_2(q)$. So $|N:A_3|=(q^6-1)/(q-1)$ (which is divisible by $q+1$). However $|N:A_2|=|S_3:B_3|$ and $|N:A_3|=|S_3:B_3'|$ are coprime indices of maximal subgroups of $\PSL_2(q)$, and by \cite[Table 7]{DGPS}, one of these indices is the index of a parabolic, and hence
 equal to $q+1$, a contradiction. Thus $\mu(N)=2$. 
\end{proof}

\section{Sporadic groups}\label{sporadicsec}
\begin{proof}[Proof of Theorem~A for the sporadic groups]
The group $T$ is one of the $27$ sporadic simple groups (note that we
did not consider the simple group ${^2F}_4(2)'$ in
Section~\ref{exceptionalLie}).  Fix $\omega\in \Omega$.  
Since $T_\omega$ is pseudo-maximal in $T$, there exists an almost
simple group $A$ with socle $T$ and a maximal subgroup $M$ of $A$ with $T\nsubseteq
M$ and $T_\omega=T\cap M$. If $T$ is not the Fisher-Griess Monster, in
the proof of this result we may use  the complete list of the maximal
subgroups of $A$ available  in~\cite{ATLAS2,ATLAS}; in particular, the
tuple $(A,T,M,T_\omega)$ is in~\cite{ATLAS2,ATLAS}. If $T$ is the
Monster, then $\Out(T)=1$, $T$ has $43$ known conjugacy classes of
maximal subgroups and, by~\cite{2}, an unknown maximal subgroup of $T$
is almost simple. In particular, if $T$ is the Monster and $T_\omega$
is conjugate to one of these unknown maximal subgroups, then by
Proposition~\ref{emb1} we have that $T$ has at most two non-trivial coprime
subdegrees. This shows that in the rest of this proof we can simply
use the information on the subgroup lattice of the sporadic groups
in~\cite{ATLAS2,ATLAS}, including the Monster.  We use the notation
in~\cite{ATLAS}. 

In order to avoid a long list of cases to consider, we have checked
with~\texttt{magma} that this theorem holds true for $|\Omega|\leq
2000$  by a direct inspection (all primitive permutation groups of
degree at most $2000$ are in the \texttt{PrimitiveGroups}
database). From the ``Specification Structure'' column in the list of
maximal subgroups of $A$ in~\cite{ATLAS}, it can be readly cheched
whether Proposition~\ref{emb1} or~\ref{emb2} applies, in this case
the theorem immediately follows. Moreover, from  the ``Specification
Order'' column, it is immediate to see whether $T_\omega$ is a
$\{p,q\}$-group, from which the theorem follows from
Lemma~\ref{id2}. Furthermore, from the ``Specification Abstract''
column, sometimes it can be easily inferred whether $T_\omega$
contains the normalizer of a Sylow $p$-subgroup of $T$, for some prime $p$,  so the theorem
follows from Proposition~\ref{sylow} in this case. (For instance, if
$T=J_1$ and 
$T_\omega\cong 7:6$, then we see that $T_\omega$ is the normalizer of
a cyclic group of order $7$. Since a Sylow $7$-subgroup of $T$ has
order $7$, we obtain that $T_\omega$ contains the normalizer of a
Sylow $7$-subgroup. For later reference we give another example. If $T=McL$ and
$T_\omega=5_{+}^{1+2}:3:8$, we see that $T_\omega$ contains a Sylow
$5$-subgroup $P$ of $T$ and that $P\unlhd T$. As $T_\omega$ is a maximal
subgroup of $T$, we obtain $T_\omega=\N T P$.) Now the proof is a case-by-case analysis on the
tuples $(A,T,M,T_\omega)$ which do not meet any of the conditions
described in this paragraph: Table~\ref{table3} gives all possible
such pairs $(T,T_\omega)$. 

Let $(T,T_\omega)$ be one of the pairs in Table~\ref{table3} and let
$N$ be the last term of the derived series of $T_\omega$ (as we defined in Section~\ref{aux}, this means that
$N\unlhd T_\omega$, $T_\omega/N$ is soluble and $N=[N,N]$). Note that
$N>1$. From Lemma~\ref{derivedSeries}, the group $N$ fixes only the
point $\omega$ of $\Omega$.

Assume that $(T,T_\omega)$ is not one of the following nine pairs.

\[
\begin{array}{lclcl}
 (J_4,2^{3+12}.(S_5\times L_3(2))),&&(Fi_{24}',2^{3+12}.(L_3(2)\times
  A_6),&&(M,2^{3+6+12+18}.(3S_6\times L_3(2))),\\ 
(B,[2^{35}].(S_5\times L_3(2))),&&(McL,2^4:A_7),
  &&(Fi_{23},2^{6+8}.(A_7\times S_3)),\\ 
(Co_2,3_{+}^{1+4}:2_{-}^{1+4}.S_5),&&
(B,3_{+}^{1+8}:2_{-}^{1+6}.U_4(2).2),&& (B,5_{+}^{1+4}:2_{-}^{1+4}.A_5.4).\\
\end{array}
\] With a direct inspection, we see that $N$ contains a normal
$p$-subgroup $P$ such that $N/P$ is either a quasisimple group, or isomorphic to  $A_5\times A_5$ and $p=2$ (here $T$ is the Harada-Norton group
$HN$). We show that $\mu(N)\leq 2$, 
from which it follows by Lemma~\ref{id2} that $T$ has at most two
non-trivial coprime subdegrees. We argue by contradiction and we
assume that $\mu(N)\geq 3$ and we let $A_1,A_2,A_3$ be three distinct
maximal subgroups of $N$ with pairwise coprime index in $N$. Let $U/P$ be the
centre of $N/P$. From
Lemma~\ref{lucky}, $p$ divides $|N:A_3|$ and $N/U=(A_1/U)(A_2/U)$ is a
maximal coprime factorization of $N/U$. Suppose first that
$N/U=A_5\times A_5$. Since in every coprime factorization of
$A_5\times A_5$, one of the two maximal subgroups has even index and
as $p=2$, we obtain  a contradiction. Suppose now that $N/U$ is
simple. Therefore $(N/U,A_1/U,A_2/U)$ is in~\cite[Table~1]{DGPS}. A
direct inspection on $N/U$, on $p$, and on the maximal coprime
factorizations of $N/U$ in~\cite[Table~1]{DGPS} shows that $p$ divides
either $|N:A_1|$  or $|N:A_2|$, a contradiction. 

It remains to consider the case that $(T,T_\omega)$ is one of the nine
pairs that we excluded above. Suppose that $(T,T_\omega)$ is one of
the first six pairs (those in the first two rows). It is immediate, comparing the order of
$T_\omega$ with the order of $T$, to check that $T_\omega$ contains
a Sylow $2$-subgroup $S$ of $T$. It can be readily seen
from~\cite[Theorem~$3$]{Ko} that $S$ is self-normalizing in $T$, that
is, $S=\N T S$. In particular, $T_\omega$ contains the normalizer of a
Sylow $2$-subgroup of $T$ and hence, from Proposition~\ref{sylow},
every non-trivial subdegree of $T$ is even.

Assume that $T=B$ and
$T_\omega=5_{+}^{1+4}:2_{-}^{1+4}.A_5.4$. From~\cite[Section~$3$ and
  Table~III]{Wilson}, 
we see that $T_\omega$ contains the normalizer of a Sylow $5$-subgroup
of $T$ and hence, by Proposition~\ref{sylow}, every non-trivial
subdegree of $T$ is divisible by $5$. 

Assume that $T=B$ and
$T_\omega=3_{+}^{1+8}:2_{-}^{1+6}.U_4(2).2$. Write
$N=[T_\omega,T_\omega]$. From Lemma~\ref{derivedSeries}, $N$ fixes only the point $\omega$ of $\Omega$. Next we show that
$\mu(N)=2$, from which the theorem follows from Lemma~\ref{id2}. We
argue by contradiction and we assume that $\mu(N)\geq 3$ and we let
$A_1,A_2$ and $A_3$ be three maximal subgroups of $N$ with pairwise
coprime index in $N$. Let $U$ be the normal subgroup of $N$ with
$N/U\cong U_4(2)$. Note that if $U\nsubseteq A_i$, then $N=A_iU$ and
 $|N:A_i|=|U:(U\cap A_i)|$ is divisible by $2$ or $3$. Since $U$ is
a $\{2,3\}$-group, there exist at most two elements of
$\{A_1,A_2,A_3\}$ not containing $U$. Moreover, if $U\leq A_i$, then
$A_i/U$ is a maximal subgroup of the simple group $N/U$. As
$\mu(U_4(2))\leq 2$ from Lemma~\ref{emb1}, there exist at most two
elements of $\{A_1,A_2,A_3\}$ containing $U$. Therefore, relabelling the set 
$\{A_1,A_2,A_3\}$ if necessary, we have two cases to consider $(i):$
$U\leq A_1,A_2$ and $U\nsubseteq A_3$, or $(ii):$ $U\leq A_1$ and
$U\nsubseteq A_2,A_3$. In~$(i)$, $N/U=(A_1/U)(A_2/U)$ is a maximal
factorization of $N/U$ with two subgroups having coprime index. With a direct
inspection on the subgroup lattice of $U_4(2)$ (or
from~\cite[Table~$1$]{DGPS}), we see that (replacing $A_1$ by $A_2$ if
necessary)  $|N:A_1|=27$ and $|N:A_2|=40$. Since
$|N:A_3|=2^\alpha 3^\beta$ for some $\alpha$ and $\beta$, we obtain a
contradiction. In~$(i)$, replacing $A_2$ 
by $A_3$ if necessary, we may assume that $|N:A_2|$ is divisible by
$2$ and $|N:A_3|$ is divisible by $3$. Now $A_1/U$ is a maximal
subgroup of $N/U$. With a direct inspection on the subgroup lattice of
$U_4(2)$ we see that $|N:A_1|\in \{27,36,40,45\}$. Since each of these
numbers is divisible by $2$ or by $3$, we obtain a contradiction.

It remains to consider the case that $T=Co_2$ and
$T_\omega=3_{+}^{1+4}:2_{-}^{1+4}.S_5$. Comparing the order of $T$
with the order of $T_\omega$ we see that $T_\omega$ contains a Sylow
$3$-subgroup $S$ of $T$. From~\cite[Section~$2$ and 
  Table~I]{Wilson}, 
we see that $|\N T S|=32|S|$. The generators of $T_\omega$
are available in~\cite{ATLAS2}. Now by a computation in
\texttt{magma} we check that $|\N {T_\omega} S |=32|S|$ and hence
$T_\omega$ contains the normalizer of a Sylow $3$-subgroup of $T$. By
Proposition~\ref{sylow}, every non-trivial subdegree of $T$ is
divisible by $3$.  
\end{proof}
 
\begin{table}[!H]
\begin{tabular}{|c|l|}\hline
Group $T$&Subgroup $T_\omega$\\\hline
$HS$ &$2^4.S_6$, $4^3:L_3(2)$, $4.2^4:S_5$\\
$J_3$&$2^4:(3\times A_5)$, $2_{-}^{1+4}:A_5$\\
$M_{24}$&$2^6:(L_3(2)\times S_3)$\\
$McL$&$3_{+}^{1+4}.2S_5$, $3^4:M_{10}$, $2^4:A_7$\\
$He$&$2^6:3.S_6$, $2_{+}^{1+6}L_3(2)$, $7^2:2L_2(7)$\\
$Ru$&$2^{3+8}:L_3(2)$, $2.2^{4+6}:S_5$, $5^2:4S_5$\\
$Suz$&$2_{-}^{1+6}:U_4(2)$, $3^5:M_{11}$, $2^{4+6}:3A_6$, $2^{2+8}:(A_5\times S_3)$\\
$O'N$&$4^3:L_3(2)$\\
$Co_3$&$3^5:(2\times M_{11})$, $3_{+}^{1+4}:4S_6$, $2^4.A_8$\\
$Co_2$&$2^{10}:M_{22}:2$, $2_{+}^{1+8}:S_6(2)$, $(2_{+}^{1+6}\times 2^{4}).A_8$, $2^{4+10}(S_5\times S_3)$, $3_{+}^{1+4}:2_{-}^{1+4}.S_5$\\
$Fi_{22}$&$3^5:(2\times U_4(2))$, $2^{10}:M_{22}$, $2^6:S_6(2)$, $(2\times 2_{+}^{1+8}:U_4(2)):2$, $(2_{+}^{5+8}:(S_3\times A_6))$\\
$HN$&$2_{+}^{1+8}.(A_5\times A_5).2$, $2^6.U_4(2)$, $2^3.2^2.2^6.(3\times L_3(2))$, $5^2.5.5^2:4A_5$, $3_{+}^{1+4}:4A5$\\
$Ly$&$5^3.L_3(5)$, $5_{+}^{1+4}:4S_6$, $3^5:(2\times M_{11})$, $3^{2+4}:(2A_5).D_8$\\
$Th$&$2^5.L_5(2)$, $2^{1+8}.A_9$, $3^5:2S_6$, $5^2:\GL_2(5)$\\
$Fi_{23}$&$2^{11}.M_{23}$, $3^{10}.(L_3(3)\times 2)$, $(2^2\times 2_{+}^{1+8}).(3\times U_4(2)).2$, $2^{6+8}:(A_7\times S_3)$\\
$Co_1$&$2^{11}.M_{24}$, $2_{+}^{1+8}.O_8^+(2)$, $2^{2+12}:(A_8\times S_3)$, $5^2:4A_5$, $2^{4+12}.(S_3\times 3S_6)$, $3^6:2M_{12}$,\\
& $3_{+}^{1+4}:2U_4(2):2$, $5_{+}^{1+2}:\GL_2(5)$, $5^3:(4\times A_5).2$,\\
$J_4$&$2^{11}:M_{24}$, $2^{10}:L_5(2)$, $2_{+}^{1+12}.3M_{22}:2$, $2^{3+12}.(S_5\times L_3(2))$\\
$Fi_{24}'$&$3^7.O_7(3)$, $3^{1+10}:U_5(2):2$, $2^{11}.M_{24}$, $2^{1+12}:3.U_4(3).2$, $3^{2+4+8}.(A_5\times 2A_4).2$,\\
& $[3^{13}]:(L_3(3)\times 2)$, $2^{3+12}.(L_3(2)\times A_6)$, $2^{6+8}.(S_3\times A_8)$\\
$B$&$2^{1+22}.Co_2$, $2^{9+16}.S_8(2)$, $2^{2+10+20}(M_{22}:2\times S_3)$, $[2^{30}].L_5(2)$, $[2^{35}].(S_5\times L_3(2))$,\\
& $3^{1+8}.2^{1+6}.U_4(2).2$, $5^3.L_3(5)$, $5^{1+4}.2^{1+4}.A_5.4$\\
$M$&$2^{1+24}.Co_1$, $2^{10+16}.O_{10}^+(2)$, $2^{2+11+22}.(M_{24}\times S_3)$, $3^{1+12}.2Suz.2$,\\
& $2^{5+10+20}.(S_3\times L_5(2))$, $2^{3+6+12+18}.(3S_6\times L_3(2))$, $3^8.O_8^-(3).2$,\\
& $3^{2+5+10}.(M_{11}\times 2S_4)$, $3^{3+2+6+6}:(L_3(3)\times SD_{16})$, $5^{1+6}:2J_2:4$,\\
& $5^{3+3}.(2\times L_3(5))$, $5^{2+2+4}:(S_3\times \GL_2(5))$, $7^{1+4}:(3\times 2S_7)$,\\
& $5^4:(3\times 2L_2(25)):2$, $7^{2+1+2}:\GL_2(7)$, $13^2:2L_2(13):4$, $7^2:\SL_2(7)$\\\hline
\end{tabular}
\caption{Pseudo-maximal subgroups of $T$ relevant to the proof of Theorem~A for the sporadic simple groups}\label{table3}
\end{table}

\end{document}